\documentclass[oneside,11pt]{amsart}
\usepackage{amsmath}
\usepackage{amssymb}
\usepackage{graphicx}
\usepackage{caption}
\usepackage{subcaption}
\usepackage{pinlabel}
\usepackage{tikz}
\usepackage{soul}
\usepackage[colorlinks=true]{hyperref}

\newtheorem{propappend}{Proposition}
\newtheorem{lemappend}[propappend]{Lemma}
\newtheorem{corappend}[propappend]{Corollary}
\newtheorem{exmpappend}[propappend]{Example}

\newtheorem{thm}{Theorem}[section]
\newtheorem{prop}[thm]{Proposition}
\newtheorem{lem}[thm]{Lemma}
\newtheorem{cor}[thm]{Corollary}

\theoremstyle{definition}
\newtheorem{defn}[thm]{Definition}
\newtheorem{rem}[thm]{Remark}

\newtheorem{ques}[thm]{Question}

\renewcommand{\bar}[1]{\overline{#1}}

\renewcommand{\emptyset}{\varnothing}

\renewcommand{\setminus}{-}

\newcommand{\field}[1]{\mathbb{#1}}
\newcommand{\Z}{\field{Z}}

\newcommand{\R}{\field{R}}

\newcommand{\E}{\field{E}}

\newcommand{\PP}{\field{P}}

\renewcommand{\implies}{\Rightarrow}

\DeclareMathOperator{\CAT}{CAT}

\title[Quasiconvexity in $3$--manifold groups]{Quasiconvexity in $3$--manifold groups}

\author{Hoang Thanh Nguyen}
\address{Beijing International Center for Mathematical Research\\
Peking University\\
 Beijing 100871, China
P.R.}
\email{nthoang.math@gmail.com}

\author{Hung Cong Tran}
\address{University of Oklahoma, Norman, OK 73019-3103, USA}
\email{Hung.C.Tran-1@ou.edu}

\author{Wenyuan Yang}
\address{Beijing International Center for Mathematical Research\\
Peking University\\
 Beijing 100871, China
P.R.}
\email{wyang@bicmr.pku.edu.cn}

\date{\today}
\begin{document}
\begin{abstract}
In this paper, we study strongly quasiconvex subgroups in a finitely generated $3$--manifold group $\pi_1(M)$. We prove that if $M$ is a compact, orientable $3$--manifold that does not have a summand supporting the Sol geometry in its sphere-disc decomposition then a finitely generated subgroup $H \le \pi_1(M)$ has finite height if and only if $H$ is strongly quasiconvex. On the other hand, if $M$ has a summand supporting the Sol geometry in its sphere-disc decomposition then $\pi_1(M)$ contains finitely generated, finite height subgroups which are not strongly quasiconvex. We also characterize strongly quasiconvex subgroups of graph manifold groups by using their finite height, their Morse elements, and their actions on the Bass-Serre tree of $\pi_1(M)$. This result strengthens analogous results in right-angled Artin groups and mapping class groups. Finally, we characterize hyperbolic strongly quasiconvex subgroups of a finitely generated $3$--manifold group $\pi_1(M)$ by using their undistortedness property and their Morse elements.
\end{abstract}
\maketitle

\section{Introduction}
In geometric group theory, one method to understand the structure of a group $G$ is to investigate subgroups of $G$. Using this approach, one often investigates subgroup $H \le G$ whose geometry reflects that of $G$. Quasiconvex subgroup of a hyperbolic group is a successful application of this approach. It is well-known that quasicovex subgroups of hyperbolic groups are finitely generated and have finite height \cite{GMRS98}. The \emph{height} of a subgroup $H$ in a group $G$ is the smallest number $n$ such that for any $(n+1)$ distinct left cosets $g_{1}H, \cdots, g_{n+1}H$ the intersection $\bigcap_{i=1}^{n+1} g_{i} H g_{i}^{-1}$ is always finite. It is a long-standing question asked by
Swarup that whether or not the converse is true (see Question~1.8 in \cite{Bes}). If the converse is true, then we could characterize quasiconvex subgroup $H$ of a hyperbolic group $G$ purely in terms of group theoretic notions.

Outside hyperbolic settings, quasiconvexity is not preserved under quasi-isometry. This means that we can not define quasiconvex subgroups of a non-hyperbolic group $G$ which are independent of the choice of finite generating set for $G$. Therefore, the second author \cite{Tran2017} develops a theory of \emph{strongly quasiconvex} subgroups of an arbitrary finitely generated group. Strong quasiconvexity does not depend on the choice of finite generating set of the ambient group and it agrees with quasiconvexity when the ambient group is hyperbolic.

\begin{defn}
Let $G$ be a finitely generated group and $H$ a subgroup of $G$. We say $H$ is \emph{strongly quasiconvex} in $G$ if for every $L \ge 1$, $C \ge 0$ there is some $R = R(L,C)$ such that every $(L,C)$--quasi-geodesic in $G$ with endpoints on $H$ is contained in the $R$--neighborhood of $H$.
\end{defn}

In \cite{Tran2017}, the second author shows that strongly quasiconvex subgroups of an arbitrary finitely generated group are also finitely generated and have finite height. Therefore, it is reasonable to extend the Swarup's question to strongly quasiconvex subgroups of finitely generated groups.

\begin{ques}[Question~1.4 \cite{2017arXiv170309032C}]
\label{ques:Hung}
Let $G$ be a finitely generated group and $H$ a finitely generated subgroup. If $H$ has finite height, is $H$ strongly quasiconvex?
\end{ques}

In this paper, we answer Question~\ref{ques:Hung} for the case $G$ is a finitely generated $3$--manifold group $\pi_1(M)$. We first prove that if $M$ has the Sol geometry, then $\pi_1(M)$ contains a finitely generated, finite height subgroup which is not strongly quasiconvex. This result gives a counterexample for Question~\ref{ques:Hung}.

\begin{prop}
\label{pintro}
Let $M$ be a compact, orientable $3$--manifold. Suppose that $M$ has a summand supporting the Sol geometry in its sphere-disc decomposition. Then $\pi_1(M)$ contains finitely generated, finite height subgroups which are not strongly quasiconvex. 
\end{prop}


For the proof of Proposition~\ref{pintro}, we first consider the case that $M$ supports the Sol geometry. In this case, by passing to a double cover, we get a manifold that is a torus bundle with Anosov monodromy $\Phi$. This implies that $\pi_1(M)$ contains an abelian-by-cyclic subgroup $\field{Z}^2\rtimes_{\Phi} \field{Z}$ as a finite index subgroups. Therefore, we study all strongly quasiconvex subgroups and finite height subgroups in abelian-by-cyclic subgroups $\field{Z}^k\rtimes_{\Phi} \field{Z}$ (see Appendix A) and Proposition~\ref{pintro} in this case (i.e. $M$ supports the Sol geometry) is obtained from those results. In the general case, let $M_1, M_2, \dots, M_n$ be the pieces of $M$ under the sphere-disc decomposition such that $M_{i_0}$ supports the Sol geometry for some $i_{0} \in \{1, \dots, n\}$. By the early observation, we can choose a finitely generated, finite height subgroup $A$ of $\pi_1(M_{i_0})$ such that $A$ is not strongly quasiconvex in $\pi_1(M_{i_0})$. The subgroup $A$ actually has finite height in $\pi_1(M)$, but it is not strongly quasiconvex in $M$.

When $M$ does not have a summand supporting the Sol geometry in its sphere-disc decomposition, we have the following theorem.
\begin{thm}
\label{thm:introduction2}
Let $M$ be a compact, orientable $3$--manifold that does not have a summand supporting the Sol geometry in its sphere-disc decomposition. Then a finitely generated subgroup $H$ in $\pi_1(M)$ has finite height if and only if $H$ is strongly quasiconvex in $\pi_1(M)$.
\end{thm}

It is well-known that if $F_2$ is the free group of rank $2$, then the free-by-cyclic group $G = F_2 \rtimes_{\phi} \Z$ (for some automorphism $\phi$ from $F_2$ to $F_2$) is the fundamental group of a $3$--manifold $M$, that is the mapping torus of the compact connected surface $\Sigma$ with one circle boundary and one genus. The manifold $M$ does not support Sol geometry, so Theorem~\ref{thm:introduction2} has the following corollary.

\begin{cor}
\label{cor:introduction1}
Let $G = F_{2} \rtimes_{\phi} \Z$ be a free-by-cyclic group, where $F_2$ is a free group of rank $2$. Let $H$ be a finitely generated subgroup of $G$. Then $H$ has finite height in $G$ if and only if $H$ is strongly quasiconvex.
\end{cor}

If $\phi \colon F_n \to F_n$ is \emph{geometric} (i.e. $\phi$ is induced from a homeomorphism $f \colon \Sigma \to \Sigma$ of some compact surface $\Sigma$), then the free-by-cyclic group $F_{n} \rtimes_{\phi} \Z$ is the fundamental group of a $3$--manifold $M$ which does not support the Sol geometry. Therefore, all finitely generated finite height subgroups of $G = F_{n} \rtimes_{\phi} \Z$ are strongly quasiconvex by Theorem~\ref{thm:introduction2}. However, not all group automorphism $\phi \colon F_n \to F_n$ is geometric (see \cite{MR1254309}). Therefore, the following question may be interesting. 

\begin{ques}
Let $\phi \colon F_n \to F_n$ be a non-geometric automorphism. Let $H$ be a finitely generated finite height subgroup of $F_{n} \rtimes_{\phi} \Z$. Is $H$ strongly quasiconvex in $F_{n} \rtimes_{\phi} \Z$?
\end{ques}


We now briefly discuss the proof of Theorem~\ref{thm:introduction2}. We assume that $H$ has finite height in $\pi_1(M)$ and we will prove that $H$ is a strongly quasiconvex subgroup. 
If $M$ has empty or tori boundary, we call $M$ is a \emph{geometric} manifold if its interior admits geometric structures in the sense of Thurston, that are $S^3$, $\mathbb{E}^3$, $\mathbb{H}^3$, $S^{2} \times \R$, $\mathbb{H}^{2} \times \R$, $\widetilde{SL}(2, \R)$, Nil and Sol. If $M$ is not geometric, $M$ is called \emph{nongeometric} $3$--manifold. When $M$ is geometric $3$--manifold, the proof is relatively easy. We refer the reader to Section~\ref{subsection:fggeometricmanifold} for details.

By Geometrization Theorem, a nongeometric $3$--manifold can be cut into hyperbolic and Seifert fibered “pieces” along a JSJ decomposition. It is called a \emph{graph manifold} if all the pieces are Seifert fibered spaces, otherwise it is a \emph{mixed manifold}. If $M$ is a mixed $3$--manifold, then $\pi_1(M)$ is relatively hyperbolic with respect to the fundamental groups of maximal graph manifold components, isolated Seifert components, and isolated JSJ tori (see \cite{Dah03} or \cite{BW13}). Therefore, we first study strongly quasiconvex subgroups and finite height subgroups in graph manifold groups and obtain Theorem~\ref{thm:introduction2} in this case. Then we use Theorem~\ref{relhyp111} which characterizes strongly quasiconvex subgroups of relatively hyperbolic groups to obtain Theorem~\ref{thm:introduction2} for the case of mixed manifold $M$. We note that the undistortedness of the subgroup $H$ must be the first step before using the above strategy. We use the recent work of Sun and the first author (see \cite{NS19}) to prove the undistortedness of $H$.

We note that the compact, connected, orientable, irreducible and $\partial$--irreducible manifold $M$ could have boundary components that are higher genus surfaces. If we are in this situation, we use the filling argument as in \cite{Sun18} to get a mixed $3$--manifold $N$ (resp. hyperbolic $3$--manifold) if $M$ has nontrivial torus decomposition (resp. $M$ has trivial torus decomposition) such that $M$ is a submanifold of $N$ with some special properties. We show that $\pi_1(M)$ is strongly quasiconvex in $\pi_1(N)$. Then we use Proposition \ref{p0} to show that $H$ also has finite height in $\pi_1(N)$. By the previous paragraph, we have that $H$ is strongly quasiconvex in $\pi_1(N)$. Then we can conclude that $H$ is strongly quasiconvex in $\pi_1(M)$ by Proposition~4.11 in \cite{Tran2017}


The most difficult part of this paper is to study strongly quasiconvex subgroups of graph manifold groups. 

\begin{thm}
\label{coolforgraphmanifold}
Let $M$ be a graph manifold. Let $H$ be a nontrivial, finitely generated subgroup of infinite index of $\pi_1(M)$. Then the following are equivalent:
\begin{enumerate}
    \item $H$ is strongly quasiconvex;
    \item $H$ has finite height in $\pi_1(M)$;
    \item All nontrivial group elements in $H$ are Morse in $\pi_1(M)$;
    \item The action of $H$ on the Bass-Serre tree of $M$ induces a quasi-isometric embedding of $H$ into the tree.
\end{enumerate}
Moreover, $H$ is a free group if some (any) above condition holds.
\end{thm}

Theorem~\ref{coolforgraphmanifold} can be compared with the work in \cite{MR3192368,KMT,Tran2017,Genevois2017} which study purely loxodromic subgroups in right-angled Artin groups and the work in \cite{BBL2016,MR3426695,Kim2017} which study convex cocompact subgroups in mapping class groups. We proved Theorem~\ref{coolforgraphmanifold} by showing that $(1)\implies (2) \implies (3) \implies (1)$ and $(3)\iff (4)$. The heart parts of Theorem~\ref{coolforgraphmanifold} are the implications $(3) \implies (1)$ and $(3)\implies (4)$. We note that an infinite order group element $g$ in a finitely generated group $G$ is \emph{Morse} if the cyclic subgroup $\langle g \rangle$ is strongly quasiconvex in $G$. For the proof of the implication $(3) \implies (1)$ and $(3) \implies (4)$, we consider the covering space $M_H \to M$ corresponding to the subgroup $H \le \pi_1(M)$, and then construct a \emph{Scott core} $K$ of $M_H$ (i.e. a compact codim--$0$ submanifold such that the inclusion of the submanifold into $M_H$ is a homotopy equivalence) as in \cite{NS19}. Take the universal cover $\tilde{M}$ of $M$ and take the preimage of $K$ in $\tilde{M}$ to get $\tilde{K} \subset \tilde{M}$.
Then we show that that $\tilde{K} \subset \tilde{M}$ is contracting subset in the sense of Sisto \cite{Sisto18}, and thus $\tilde{K}$ is strongly quasiconvex in $\tilde{M}$. As a consequence $H$ is strongly quasiconvex in $\pi_1(M)$. Moreover, special properties of the construction of Scott core $K$ above also allow us to get the implication $(3)\implies (4)$. We refer the reader to Section~\ref{sscfgmg} for the full proof of Theorem~\ref{coolforgraphmanifold}.

We also generalize a part of Theorem~\ref{coolforgraphmanifold} to characterize hyperbolic strongly quasiconvex subgroups in finitely generated $3$--manifold groups.

\begin{thm}
\label{coolformanifold}
Let $M$ be a $3$--manifold with finitely generated fundamental group. Let $H$ be an undistorted subgroup of $\pi_1(M)$ such that all infinite order group elements in $H$ are Morse in $G$. Then $H$ is hyperbolic and strongly quasiconvex. 
\end{thm}

In contrast to the case of graph manifold, a subgroup whose all infinite order elements are Morse in a finitely generated $3$--manifold group can not be automatically strongly quasiconvex without the hypothesis undistortedness. For example, if $M$ is a closed hyperbolic $3$--manifold and $H$ is a finitely generated geometrically infinite subgroup of $\pi_1(M)$, then all infinite order group elements in $H$ are Morse in $\pi_1(M)$. However, $H$ is not strongly quasiconvex because $H$ is exponentially distorted in $\pi_1(M)$ by the Covering Theorem (see \cite{Canary96}) and the Subgroup Tameness Theorem (see \cite{Agol04} and \cite{CG06}).







We now discuss the proof of Theorem~\ref{coolformanifold}. First, we call a finitely generated subgroups $H$ of a finitely generated group $G$ \emph{purely Morse} if all infinite order elements in $H$ are Morse in $G$. By some standard arguments, we first reduce to the case where $M$ is compact, connected, orientable, irreducible and $\partial$--irreducible. When $M$ is a geometric $3$--manifold, the proof is relatively easy. When $M$ is a nongeometric $3$--manifold, then $M$ is either a graph manifold or a mixed manifold. In the first case, the proof is already given in Theorem~\ref{coolforgraphmanifold}. In the latter case, we note that $\pi_1(M)$ is relatively hyperbolic with respect to the fundamental groups of maximal graph manifold components, isolated Seifert components, and isolated JSJ tori (see \cite{Dah03} or \cite{BW13}). Thus, we first prove that if all peripheral subgroups of a relatively hyperbolic group have the property that all their undistorted purely Morse subgroups are hyperbolic and strongly quasiconvex, then so does the ambient group (see Corollary~\ref{corcoolcool}). This implies that we can prove Theorem~\ref{coolformanifold} for the case of mixed $3$--manifold by using Theorem~\ref{coolforgraphmanifold}.

If the manifold $M$ has at least a boundary component that is higher genus surface, then this case follows from a similar filling argument as in the proof of Theorem~\ref{thm:introduction2} (although some details are different).

\subsection*{Overview}
In Section~\ref{sec:background}, we review concepts finite height, strongly quasiconvex and stable subgroups in finitely generated groups. A proof of Theorem~\ref{coolforgraphmanifold} is given in Section~\ref{sscfgmg}. In Section~\ref{sec:sqcinfgmanifold}, we give complete proofs of Proposition~\ref{pintro}, Theorem~\ref{thm:introduction2}, and Theorem~\ref{coolformanifold}. In Appendix~A, we study strongly quasiconvex subgroups and finite height subgroups of abelian-by-cyclic groups $\field{Z}^k\rtimes_{\Phi} \field{Z}$.

\subsection*{Acknowledgments}
The authors are grateful for the insightful
and detailed critiques of the referee that have helped improve the exposition
of this paper. The authors especially appreciate the referee for pointing out
a mistake in Theorem~1.4 in the earlier version.
H. T. was supported by an AMS-Simons Travel Grant. W. Y. is supported by the National Natural Science Foundation of China (No. 11771022).

\section{Preliminaries}
\label{sec:background}
In this section, we review concepts finite height subgroups, strongly quasiconvex subgroups, stable subgroups, Morse elements, purely Morse subgroups, and their basic properties that will be used in this paper.

\subsection{Finite height subgroups, malnormal subgroups, and their properties}
\begin{defn}
Let $G$ be a group and $H$ a subgroup of $G$. Then 
\begin{enumerate}
\item Conjugates $g_1Hg_1^{-1}, \cdots g_kHg_k^{-1} $ are \emph{essentially distinct} if the cosets $g_1H,\cdots,g_kH$ are distinct.
\item $H$ has height at most $n$ in $G$ if the intersection of any $(n+1)$ essentially distinct conjugates is finite. The least $n$ for which this is satisfied is called the \emph{height} of $H$ in $G$.
\item $H$ is \emph{almost malnormal} in $G$ if $H$ has the height at most $1$ in $G$. 
\item $H$ is \emph{malnormal} in $G$ if for each $g\in G-H$ the subgroup $gHg^{-1}\cap H$ is trivial.
\end{enumerate}
We observe that a malnormal subgroup is always almost malnormal. Moreover, if $G$ is a torsion-free group, then every almost malnormal subgroup of $G$ is malnormal.
\end{defn}

The following proposition provides some basic properties of finite height subgroups. We will use these properties many times for studying finite height subgroups of $3$--manifold groups.

\begin{prop}
\label{p0}
Let $G$ be a group and $H$ a subgroup. Then:
\begin{enumerate}
    \item If $H$ has finite height in $G$ and $G_1$ is a subgroup of $G$, then $H\cap G_1$ has finite height in $G_1$.
    \item If $G_1$ is a finite index subgroup of $G$ and $H\cap G_1$ has finite height in $G_1$, then $H$ has finite height in $G$.
    \item If $H$ is a finite height subgroup of $G$ and $K$ is a finite height subgroup of $H$, then $K$ has finite height in $G$.
    \item If $H_1$, $H_2$ are two finite height subgroup of $G$, then $H_1\cap H_2$ has finite height in $G$, $H_1$, and $H_2$.
\end{enumerate}
\end{prop}

\begin{proof}
We first prove Statement (1). Assume that the height of $H$ in $G$ is at most $n$. Then the intersection of any $(n+1)$ essentially distinct conjugates of $H$ in $G$ is finite. Let $H_1=H\cap G_1$ and let $g_1H_1, g_2H_1, \cdots g_nH_1, g_{n+1}H_1$ be $(n+1)$ distinct left cosets of $H_1$ in $G_1$. It is straight forward that $g_iH\neq g_jH$ for $i\neq j$. Therefore, $\cap g_iHg^{-1}_i$ is finite. This implies that $\cap g_iH_1g^{-1}_i$ is also finite. Therefore, the height of $H_1$ in $G_1$ is also at most $n$. 

We now prove Statement (2). Assume the index of $G_1$ in $G$ is $k$. Let $H_1=H\cap G_1$. Since $H_1$ has finite height in $G_1$, there is a number $m$ such that the intersection of any $(m+1)$ essentially distinct conjugates of $H_1$ in $G_1$ is finite. Let $n=km$ and we will prove that the height of $H$ in $G$ is at most $n$. In fact, let $\mathcal{A}=\{g_1H,g_2H,\cdots, g_nH, g_{n+1}H\}$ be a collection of $(n+1)$ distinct left cosets of $H$ in $G$. Then there is $(m+1)$ left cosets in $\mathcal{A}$ (called $g_{\ell(1)}H, g_{\ell(2)}H,\cdots, g_{\ell(m)}H, g_{\ell(m+1)}H$ such that $g_{\ell(i)}$ lies in the same left coset $gG_1$. Therefore, $g_{\ell(i)}=gk_i$ for some $k_i \in G_1$. It is straightforward that $k_iH_1\neq k_jH_1$ for $i\neq j$. Therefore, $\cap k_iH_1k^{-1}_i$ is finite. Since $H_1$ is of finite index in $H$, the intersection $\cap k_iHk^{-1}_i$ is also finite. Therefore, $\cap g_{\ell(i)}Hg^{-1}_{\ell(i)}=g\bigl(\cap k_iHk^{-1}_i\bigr)g^{-1}$ is finite. This implies that the height of $H$ in $G$ is at most $m$.

We now prove the Statement (3). Assume that the height of $H$ in $G$ is at most $n$ and the height of $K$ in $H$ is at most $m$. Let $k=mn+1$ We will prove that the height of $K$ in $G$ is at most $k$. Let $g_1K, g_2K, \cdots g_kK, g_{k+1}K$ be $(k+1)$ distinct left cosets of $K$ in $G$. If $\mathcal{A}=\{g_1H,g_2H,\cdots, g_kH, g_{k+1}H\}$ contains more than $n$ distinct left cosets of $H$ in $G$, then $\cap g_iHg^{-1}_i$ is finite and therefore $\cap g_iKg^{-1}_i$ is also finite. Otherwise, $\mathcal{A}=\{g_1H,g_2H,\cdots, g_kH, g_{k+1}H\}$ contains at most $n$ distinct left cosets of $H$ in $G$ and therefore there is a group element $g$ in $G$ and $(m+1)$ distinct elements $g_{\ell(1)}, g_{\ell(2)},\cdots, g_{\ell(m)}, g_{\ell(m+1)}$ in $\{g_1,g_2,\cdots,g_k\}$ such that $g_{\ell(i)}=gh_i$ for some $h_i\in H$. Since the height of $K$ in $H$ is at most $m$ and $\mathcal{B}=\{h_1K,h_2K,\cdots, h_mK, h_{m+1}K\}$ is a collection of $(m+1)$ distinct left cosets of $K$ in $H$, the intersection $\cap h_iKh^{-1}_i$ is finite. This implies that $\cap g_{\ell(i)}Kg^{-1}_{\ell(i)}=g\bigl(\cap h_iKh^{-1}_i\bigr)g^{-1}$ is also finite and therefore $\cap g_iKg^{-1}_i$ is finite. Thus $K$ has finite height in $G$. Statement (4) is obtained from Statement (1) and Statement (3). 
\end{proof}

Finite subgroups and finite index subgroups always have finite height in the ambient group. On the other hand, the following proposition provides certain groups whose finite height subgroups are either finite or has finite index in the ambient groups. This proposition will help us study finite height subgroups of almost all geometric manifold groups (except hyperbolic manifold groups and Sol manifold groups) and Seifert manifold groups. 

\begin{prop}
\label{central}
Let $G$ be a group such that the centralizer $Z(G)$ of $G$ is infinite. Let $H$ be a finite height infinite subgroup of $G$. Then $H$ must have finite index in $G$.
\end{prop}

\begin{proof}
We first assume that $Z(G) \cap H$ has infinite index in $Z(G)$. Then there is an infinite sequence $(t_n)$ of elements in $Z(G)$ such that $t_i(Z(G) \cap H)\neq t_j(Z(G) \cap H)$ for $i\neq j$. Therefore, it is straightforward that $t_iH\neq t_jH$ for $i\neq j$. Also, $\bigcap t_nHt_n^{-1}=H$ is infinite. This contradicts to the fact that $H$ has finite height. Therefore, $Z(G) \cap H$ has finite index in $Z(G)$, In particular, $Z(G) \cap H$ is infinite. Assume that $H$ has infinite index in $G$. Then there is an infinite sequence $\{g_nH\}$ of distinct left cosets of $H$. However, $\bigcap g_nHg_n^{-1}$ is infinite since it contains the infinite subgroup $Z(G) \cap H$. This contradicts to the fact that $H$ has finite height. Therefore, $H$ must have finite index in $G$.
\end{proof}

We now discuss how a finite height subgroup interacts with a normal subgroup with certain properties (see Corollary~\ref{cor0}). This result will be used to study finite height subgroups in abelian-by-cyclic groups $\field{Z}^k\rtimes_{\Phi} \field{Z}$ in Appendix A. 

\begin{prop}[Proposition A.1 in \cite{2017arXiv170309032C}]
\label{fhimpliessq}
Let $G$ be a group and suppose there is a collection $\mathcal{A}$ of subgroups of $G$ that satisfies the following conditions:
\begin{enumerate}
\item For each $A$ in $\mathcal{A}$ and $g$ in $G$ the conjugate $g^{-1}Ag$ also belongs to $\mathcal{A}$ and there is a finite sequence $$A=A_0, A_1, \cdots, A_n=g^{-1}Ag$$ of subgroups in $\mathcal{A}$ such that $A_{j-1}\cap A_j$ is infinite for each $j$;
\item For each $A$ in $\mathcal{A}$ each finite height subgroup of $A$ must be finite or have finite index in $A$.
\end{enumerate}
Then for each infinite index finite height subgroup $H$ of $G$ the intersection $H\cap A$ must be finite for all $A$ in $\mathcal{A}$.
\end{prop}

\begin{cor}
\label{cor0}
Let $G$ be a group and $H$ a finite height subgroup of infinite index. Let $N$ be a normal subgroup of $G$ such that each finite height subgroup of $N$ must be finite or have finite index in $N$. Then the intersection $H\cap N$ must be finite.
\end{cor}

\begin{proof}
We use Proposition~\ref{fhimpliessq} for the case $\mathcal{A}$ consists of only element $N$.
\end{proof}

The following proposition studies certain property to finite height subgroups of certain graphs of groups. This proposition will be used to study finite height subgroups in graph manifold groups.

\begin{prop}
\label{prop:tree}
Assume a group $G$ is decomposed as a finite graph $T$ of groups that satisfies the following.

\begin{enumerate}
    \item For each vertex $v$ of $T$ each finite height subgroup of vertex group $G_v$ must be finite or have finite index in $G_v$.
    \item Each edge group is infinite.
\end{enumerate}
Then, if $H$ is a finite height subgroup of G of infinite index, then $H \cap gG_vg^{-1}$ is finite for each vertex group $G_v$ and each group element $g$. In particular, if $H$ is torsion-free, then $H$ is a free group.
\end{prop}

\begin{proof}
Let $\mathcal{A}$ be the collection of all conjugates of vertex groups in the decomposition of $G$. Then $\mathcal{A}$ satisfies Condition (2) in Proposition~\ref{fhimpliessq} by the hypothesis. Let $\tilde{T}$ be the Bass-Serre tree of the decomposition. Then conjugates of vertex groups (resp. edge groups) correspond to vertices (resp. edges) of $\tilde{T}$. Since $\tilde{T}$ is connected and each edge group is infinite, the collection $\mathcal{A}$ also satisfies Condition (1) in Proposition~\ref{fhimpliessq}. Therefore, if $H$ is a finite height subgroup of G of infinite index, then $H \cap gG_vg^{-1}$ is finite for each vertex group $G_v$ and each group element $g$. 

We now further assume that $H$ is torsion-free. Then $H \cap gG_{v}g^{-1}$ is trivial for each vertex group $G_v$ and each element $g \in G$. Also, we know that $G$ acts on the Bass-Serre tree $\tilde{T}$ such that the stabilizer of a vertex of $T$ is a conjugate of a vertex group. Therefore, $H$ acts freely on the Bass-Serre tree $\tilde{T}$. This implies that $H$ is a free group.
\end{proof}

\subsection{Quasiconvex subgroups, strongly quasiconvex subgroups, and stable subgroups}

We first discuss the concepts of quasiconvex subsets and strongly quasiconvex subsets in a geodesic space. These concepts are the foundation for the concepts of quasiconvex subgroups and strongly quasiconvex subgroups in a finitely generated group.

\begin{defn} [(Strongly) quasiconvex subsets]
Let $X$ be a geodesic space and let $Y$ be a subset of $X$. The subset $Y$ is \emph{quasiconvex} in $X$ if there is a constant $D>0$ such that every geodesic with endpoints on $Y$ is contained in the $D$--neighborhood of $Y$. The subset $Y$ is \emph{strongly quasiconvex} if for every $K \geq 1,C \geq 0$ there is some $M = M(K,C)$ such that every $(K,C)$--quasi–geodesic with endpoints on $Y$ is contained in the $M$--neighborhood of $Y$. 
\end{defn}

It follows directly from the definition that strong quasiconvexity is a quasi-isometry invariant in the following sense.

\begin{lem}
	\label{lem:quasiconvex_qi_invariant} 
	Let $X$ and $Z$ be a geodesic metric spaces and $f\colon X \rightarrow Z$ be a quasi-isometry. If $Y$ is a strongly quasiconvex subset of $X$, then $f(Y)$ is a strongly quasiconvex subset of $Z$.
\end{lem}

Quasiconvexity is not a quasi-isometry invariant but it is equivalent to strong quasiconvexity in the settings of hyperbolic spaces. We now define quasiconvex subgroups, strongly quasiconvex subgroups, and Morse elements in a finitely generated group.

\begin{defn}[Quasiconvex subgroups, strongly quasiconvex subgroups, and stable subgroups] 
Let $G$ be a finitely generated group and $H$ a subgroup of $G$. We say $H$ is \emph{quasiconvex} in $G$ with respect to some finite generating set $S$ of $G$ if $H$ is a quasiconvex subset in the Cayley graph $\Gamma(G,S)$. We say $H$ is \emph{strongly quasiconvex} in $G$ if $H$ is a strongly quasiconvex subset in the Cayley graph $\Gamma(G,S)$ for some (any) finite generating set $S$. We say $H$ is \emph{stable} in $G$ if $H$ is strongly quasiconvex and hyperbolic.
\end{defn}

\begin{rem}
If $H$ is a quasiconvex subgroup of a group $G$ with respect to some finite generating set $S$, then $H$ is also finitely generated and undistorted in $G$ (see Lemma 3.5 of Chapter III.$\Gamma$ in \cite{MR1744486}). However, we emphasize that the concept of quasiconvex subgroups depends on the choice of finite generating set of the ambient group.

The strong quasiconvexity of a subgroup does not depend on the choice of finite generating sets by Lemma~\ref{lem:quasiconvex_qi_invariant}. It is clear that a strongly quasiconvex subgroup is also quasiconvex with respect to some (any) finite generating set of the ambient group. In particular, each strongly quasiconvex subgroup is finitely generated and undistorted in the ambient group (see Theorem~1.2 in \cite{Tran2017}).

Finally, we would like to emphasize that the above definition of stable subgroup is equivalent to the definition originally given by Durham and Taylor in \cite{MR3426695}. We refer the reader to the work of the second author~\cite{Tran2017} to see the proof of the equivalence. 
\end{rem}

In the following theorem, we review a result proved by the second author \cite{Tran2017} that a strongly quasiconvex subgroup always has finite height. 

\begin{thm}[Theorem 1.2 in \cite{Tran2017}]
\label{tran}
Let $G$ be a finitely generated group and $H$ a strongly quasiconvex subgroup of $G$. Then $H$ is finitely generated and has finite height in $G$.
\end{thm}

\begin{defn}[Morse elements and purely Morse subgroups]
Let $G$ be a finitely generated group. A group element $g$ in $G$ is \emph{Morse} if $g$ is of infinite order and the cyclic subgroup generated by $g$ is strongly quasiconvex. A finitely generated subgroup $H$ of $G$ is \emph{purely Morse} if all infinite order elements of $H$ are Morse in $G$. 
\end{defn}

\begin{prop}\cite{MR3426695}
Let $G$ be a finitely generated group and let $H$ be a stable subgroup of $G$. Then $H$ is undistorted and purely Morse.
\end{prop}

The following proposition is a direct result of Proposition 4.10 and Proposition 4.12 in \cite{Tran2017}.

\begin{prop}
\label{pcool}
Let $G$ be a finitely generated group and let $H$ be a strongly quasiconvex subgroup of $G$. A group element $h\in H$ is Morse in $H$ if and only if it is Morse in $G$. 
\end{prop}

The following corollary is a direct result of Proposition~\ref{pcool}.

\begin{cor}
\label{ccool}
Let $G$ be a finitely generated group and let $H$ be a strongly quasiconvex subgroup of $G$. Then:
\begin{enumerate}
    \item If $G$ has the property that all purely Morse subgroups of $G$ are undistorted, then $H$ also has the property all purely Morse subgroups of $H$ are undistorted;
    \item If $G$ has the property that all purely Morse subgroups of $G$ are stable, then $H$ also has the property all purely Morse subgroups of $H$ are stable;
    \item If $G$ has the property that all undistorted purely Morse subgroup of $G$ are stable, then $H$ also has the property all undistorted purely Morse subgroup of $H$ are stable.
\end{enumerate}
\end{cor}

We now define hyperbolic elements in relatively hyperbolic groups.

\begin{defn}
Let $(G,\PP)$ be a finitely generated relatively hyperbolic group. An infinite order element $g$ in $G$ is \emph{hyperbolic} if $g$ is not conjugate to an element of a subgroup in $\PP$.
\end{defn}

The following theorem provides a characterization of a strongly quasiconvex subgroup $H$ in a relatively hyperbolic group $(G,\PP)$ in terms of its interactions with peripheral subgroups. This theorem will be used to study strongly quasiconvex subgroups in mixed manifold groups and strengthen the result to finitely generated $3$--manifold groups.

\begin{thm}[Theorem 1.9 in \cite{Tran2017}]
\label{relhyp111}
Let $(G,\PP)$ be a finitely generated relatively hyperbolic group and $H$ a finitely generated undistorted subgroup of $G$. Then the following are equivalent:
\begin{enumerate}
\item The subgroup $H$ is strongly quasiconvex in $G$.
\item The subgroup $H\cap gPg^{-1}$ is strongly quasiconvex in $gPg^{-1}$ for each conjugate $gPg^{-1}$ of peripheral subgroup in $\PP$.
\item The subgroup $H\cap gPg^{-1}$ is strongly quasiconvex in $G$ for each conjugate $gPg^{-1}$ of peripheral subgroup in $\PP$. 
\end{enumerate}
\end{thm}

The following proposition characterizes all Morse elements in a relatively hyperbolic group.

\begin{prop}
\label{Morseelementinrelhyp}
Let $(G,\PP)$ be a finitely generated relatively hyperbolic group. An infinite order element $g$ in $G$ is Morse if and only if $g$ is either a hyperbolic element or $g$ is conjugate into a Morse element in some subgroup $P$ in $\PP$.
\end{prop}

\begin{proof}
The ``only if'' direction is obtained from Proposition~\ref{pcool} and the fact that each conjugate of subgroup in $\PP$ is strongly quasiconvex in $G$. Therefore, we only need to prove the ``if'' direction. If $g=hg_0h^{-1}$ for some Morse element $g_0$ in a subgroup $P\in \PP$, then both $g$ and $g_0$ are Morse in $G$. Otherwise, $g$ is an hyperbolic element. Therefore, the cyclic subgroup $\langle g \rangle$ is undistorted (see \cite{Osin06}). We claim that $\langle g \rangle \cap uPu^{-1}$ is trivial for each subgroup $P\in \PP$ and each group element $u\in G$. In fact, if $\langle g \rangle \cap uPu^{-1}$ is not trivial for some $P\in \PP$ and some group element $u\in G$. Then, there is a positive integer $n$ such that $g^n$ in an element in $uPu^{-1}$. Let $g_1=u^{-1}g u$. Then $g_1^n$ is a group element in $P$. Since $g$ is hyperbolic, $g_1$ is not a group element in $P$. Also, $g_1Pg_1^{-1}\cap P$ is infinite. This contradicts to the fact that $P$ is almost malnormal (see \cite{Osin06}). Therefore, $\langle g \rangle \cap uPu^{-1}$ is trivial for each subgroup $P\in \PP$ and each group element $u\in G$. Therefore, the cyclic subgroup $\langle g \rangle$ is strongly quasiconvex in $G$ by Theorem~\ref{relhyp111}. Therefore, $g$ is a Morse element in $G$. 
\end{proof}

The following theorem provides a characterization of a stable subgroup $H$ in a relatively hyperbolic group $(G,\PP)$ in terms of its interactions with peripheral subgroups. 

\begin{thm}[Corollary 1.10 in \cite{Tran2017}]
\label{relhyp112}
Let $(G,\PP)$ be a finitely generated relatively hyperbolic group and $H$ a finitely generated undistorted subgroup of $G$. Then the following are equivalent:
\begin{enumerate}
\item The subgroup $H$ is stable in $G$.
\item The subgroup $H\cap gPg^{-1}$ is stable in $gPg^{-1}$ for each conjugate $gPg^{-1}$ of peripheral subgroup in $\PP$.
\item The subgroup $H\cap gPg^{-1}$ is stable in $G$ for each conjugate $gPg^{-1}$ of peripheral subgroup in $\PP$. 
\end{enumerate}
\end{thm}

\begin{cor}
\label{corcoolcool}
Let $(G,\PP)$ be a finitely generated relatively hyperbolic group. If each subgroup $P$ in $\PP$ has the property that all undistorted purely Morse subgroups in $P$ are stable, then $G$ also has the property that all undistorted purely Morse subgroups in $G$ are stable
\end{cor}

\begin{proof}
Let $H$ be an undistorted purely Morse subgroup in $G$. We will prove that $H\cap gPg^{-1}$ is stable in $gPg^{-1}$ for each conjugate $gPg^{-1}$ of peripheral subgroup $P$ in $\PP$. By using Proposition~\ref{pcool} and the fact that $gPg^{-1}$ is strongly quasiconvex in $G$ we can conclude that $H\cap gPg^{-1}$ is a purely Morse subgroup of $gPg^{-1}$. We now claim that $H\cap gPg^{-1}$ is an undistorted subgroup in $gPg^{-1}$. By Proposition 4.11 in \cite{Tran2017} and the fact that $gPg^{-1}$ is strongly quasiconvex in $G$, the subgroup $gPg^{-1}\cap H$ is finitely generated and undistorted in $H$. Also, the subgroup $H$ is undistorted in $G$. Then $gPg^{-1}\cap H$ is also undistorted in $G$. This implies that $gPg^{-1}\cap H$ is undistorted in $gPg^{-1}$. Therefore, $H\cap gPg^{-1}$ is stable in $gPg^{-1}$. Thus, $H$ is stable in $G$ by Theorem~\ref{relhyp112}. 
\end{proof}

\section{Strongly quasiconvex subgroups in graph manifold groups}
\label{sscfgmg}
Let $M$ be a compact, connected, irreducible, orientable $3$--manifold with empty or tori boundary. $M$ is called {\it geometric} if its interior admits geometric structures in the sense of Thurston, that are $S^3$, $\mathbb{E}^3$, $\mathbb{H}^3$, $S^{2} \times \R$, $\mathbb{H}^{2} \times \R$, $\widetilde{SL}(2, \R)$, Nil and Sol. 

If $M$ is not geometric, then $M$ is called a {\it nongeometric $3$--manifold}. By Geometrization of $3$--manifolds, there is a nonempty minimal union $\mathcal{T} \subset M$ of disjoint essential tori and Klein bottles, unique up to isotopy, such that each component of $M \backslash \mathcal{T}$ is either a Seifert fibered piece or a hyperbolic piece. $M$ is called {\it graph manifold} if all the pieces of $M \backslash \mathcal{T}$ are Seifert pieces, otherwise it is a {\it mixed manifold}.

We remark here that the geometric decomposition is slightly different from the torus decomposition, but they are closely related (when $M$ has no decomposing Klein bottle, then two these decompositions agree with each other). Since we only consider virtual properties of $3$--manifolds in this paper and two these decompositions agree with each other on some finite cover of $M$, such a difference can be get rid of by passing to some finite cover of $M$.


In this section, we study strongly quasiconvex subgroups in graph manifold groups. More precisely, we prove that stable subgroups, strongly quasiconvex subgroups, and finitely generated finite height subgroups in graph manifold groups are all equivalent; and we characterize these subgroups in terms of their group elements (see Theorem~\ref{coolforgraphmanifold}). We first characterize Morse elements in graph manifold groups.

\begin{prop}[Morse elements in graph manifold groups]
\label{psis}
Let $M$ be a graph manifold group. Then a nontrivial group element $g$ in $\pi_1(M)$ is Morse if and only if $g$ is not conjugate into any Seifert subgroups. 
\end{prop}

\begin{proof}
Since each Seifert subgroup is virtually a product of a free group and $\field{Z}$, then it can not contain an infinite cyclic subgroup which is strongly convex by Proposition~\ref{central} and Theorem~\ref{tran}. Therefore, if a nontrivial group element $g$ is Morse, then $g$ is not conjugate into any Seifert subgroups. On the other hand, if $g$ is not conjugate into any Seifert subgroups, then $g$ is Morse by Lemma 2.8 and Proposition 3.6 in \cite{Sisto18}. 
\end{proof}

We now talk about the proof of Theorem~\ref{coolforgraphmanifold}. The implication $(1)\implies (2)$ is obtained from Theorem~\ref{tran}. We note that $\pi_1(M)$ is decomposed as a graph of Seifert manifold groups and this decomposition satisfies conditions (1) and (2) of Proposition~\ref{prop:tree}. Therefore, the implication ``$(2) \implies (3)$'' follows from Proposition~\ref{psis}. Moreover, the fact $H$ is free when Condition (2) holds is also obtained from Proposition~\ref{prop:tree}. Therefore, we now only need to prove the implication ``$(3) \implies (1)$'' (see Proposition~\ref{pkey}) and the equivalence ``$(3)\iff (4)$'' (see Proposition~\ref{actionontree}).

\subsection{Some preparations:}
\label{subsubsection:1}
Firstly, by passing to a finite cover $M'$ of $M$, we can assume that each Seifert piece $M_{i}$ of $M$ is a product $F_{i} \times S^1$, and $M$ does not contains a twisted $I$--bundle over the Klein bottle (see \cite{PW14}). We remark here that all nontrivial group elements in a finitely generated subgroup $H$ of $\pi_1(M)$ are Morse in $\pi_1(M)$ if and only if all nontrivial group elements in $H': = H \cap \pi_1(M')$ are Morse in $\pi_1(M')$. Moreover, $H$ is strongly quasiconvex in $\pi_1(M)$ if and only if $H'$ is strongly quasiconvex in $\pi_1(M')$. Therefore, it suffices to prove the implication ``$(3) \implies (1)$'' for $H' \le \pi_1(M')$, we still denote the subgroup of
the $3$--manifold group by $H \le \pi_1(M)$.

{\bf A Scott core of $M_H$:}
Let $p \colon M_H \to M$ be the covering space corresponding to $H$. Since $M$ has nontrivial torus decomposition, $M_H$ has an induced graph of space structure. Each elevation (i.e. a component of the preimage) of a piece of $M$ in $M_H$ is called a \emph{piece} of $M_H$, and each elevation of a decomposition torus of $M$ in $M_H$ is called an \emph{edge space} of $M_H$. Since $H$ is finitely generated, there exists a finite union of pieces $M_H^c\subset M_H$, such that the inclusion $M_H^c \to M_H$ induces an isomorphism on fundamental groups, and we take $M_H^c$ to be the minimal such manifold. In \cite{NS19}, a compact Scott core $K$ of $M_H$ (and thus $\pi_1(K) =H$) has been constructed explicitly (see Preparation Step II in \cite{NS19}) and this Scott core satisfies the following properties.
\begin{enumerate}
    \item $K \subset M^{c}_{H}$ and for each piece $M_{H,i}$ of $M^{c}_{H}$, the intersection $K \cap M_{H,i}$ is a compact Scott core of $M_{H,i}$. Note that each $M_{H,i}$ covers a Seifert piece of $M$.
    \item For each edge space $E \subset M_{H}$ the intersection $K \cap E$ is either empty or a disc of $E$ if all piece $M_{H,i}$ of $M^{c}_H$ are simply connected.
\end{enumerate}

The following lemmas capture some geometric properties of subgroups of manifold groups whose all nontrivial elements are Morse.

\begin{lem}
\label{lem:M_H,isimplyconnected}
Let $M$ be a graph manifold, and let $H$ be a finitely generated purely Morse subgroup of $\pi_1(M)$. Then each piece $M_{H,i}$ of $M^{c}_{H}$ is simply connected.
\end{lem}
\begin{proof}
For each Seifert piece $M_i$ of $M$ and each group element $g\in \pi_1(M)$ the subgroup $H\cap g\pi_1(M_i)g^{-1}$ is trivial by Proposition~\ref{psis}. Since $M_{H,i}$ is a covering space of a Seifert piece $M_i$ of $M$ corresponding to subgroup $H \cap g \pi_1(M_i) g^{-1}$ for some $g \in G$, it follows that $M_{H,i}$ is simply connected.
\end{proof}

\begin{lem}
\label{lem:scottcoreuniversal}
Let $M$ be a graph manifold. Equip $M$ with a Riemannian metric and lift this metric to the universal cover $\tilde{M}$.Let $H$ be a finitely generated purely Morse subgroup of $\pi_1(M)$. Let $K \subset M^{c}_{H}$ be the Scott core of $M_H$ given by previous paragraphs.
Let $\tilde{K}$ be the preimage of $K$ in the universal cover $\tilde{M}$ of $M$. Then there exists a positive constant $\delta$ such that for any piece $\tilde{M}_{i}$ of $\tilde{M}$ with $\tilde{K} \cap \tilde{M}_{i} \neq \emptyset$, then $\tilde{K} \cap \tilde{M}_{i}$ is simply connected and $diam( \tilde{K} \cap \tilde{M}_{i}) < \delta$.
\end{lem}
\begin{proof}
For each piece $M_{H,j}$ of $M^{c}_{H}$, let $K_j = K \cap M_{H,j}$.
Since $K_j$ is a Scott core of $M_{H,j}$ and $M_{H,j}$ is simply connected (see Lemma~\ref{lem:M_H,isimplyconnected}), it follows that $K_j$ is simply connected. Since there are only finitely many pieces $M_{H,j}$ of $M^{c}_{H}$, it follows there are only finitely many $K_j$.

Let $\tilde{M}_j$ the universal cover of $M_{H,j}$, and let $\tilde{K}_{j}$ be the preimage of $K_j$ in $\tilde{M}_{j}$. Then $\tilde{K}_{j} = \tilde{K} \cap \tilde{M}_j$. Since $K_j$ is simply connected and compact, it follows that $\tilde{K}_{j}$ is compact (actually $\tilde{K}_{j}$ is homeomorphic to $K_j$ since two universal covers of a common space are homeomorphic). Since there are finitely many $K_j$ and each of them has bounded diameter. We then can find a uniform constant $\delta$ such that the statement of the lemma holds.
\end{proof}

\subsection{$\mathcal{PS}$--contracting and strong quasiconvexity in graph manifold groups}
\label{subsubsection:2}

In \cite{Sisto18}, Sisto constructed a certain collection of paths in the universal cover of graph manifolds which is called \emph{the path system $\mathcal{PS}$} and the concept of $\mathcal{PS}$--contracting to study Morse elements in graph manifold groups. We now use these concepts to study strongly quasiconvex subgroups in graph manifold groups.

\begin{defn}[path system, \cite{Sisto18}]
Let $X$ be a metric space. A \emph{path system} $\mathcal{PS}$ in $X$ is a collection of $(c,c)$--quasi-geodesic for some $c$ such that
\begin{enumerate}
    \item Any subpath of a path in $\mathcal{PS}$ is in $\mathcal{PS}$.
    \item All pairs of points in $X$ can be connected by a path in $\mathcal{PS}$.
\end{enumerate}
\end{defn}

\begin{defn}[$\mathcal{PS}$--contracting, \cite{Sisto18}]
\label{defn:contracting}
Let $X$ be a metric space and let $\mathcal{PS}(X)$ be a path system in $X$.
A subset $A$ of $X$ is called $\mathcal{PS}(X)$--contracting if there exists $C>0$ and a map $\pi \colon X \to A$ such that
\begin{enumerate}
    \item For any $x \in A$, then $d(x, \pi(x)) \le C$
    \item For any $x, y \in X$ such that $d(\pi(x), \pi(y)) \ge C$, then for any path $\gamma$ in $\mathcal{PS}(X)$ connecting $x$ to $y$ we have $d(\pi(x), \gamma) \le C$ and $d(\pi(y), \gamma) \le C$.
\end{enumerate}
The map $\pi$ will be called $\mathcal{PS}$--projection on $A$ with constant $C$.
\end{defn}

The following lemma seems well-known to experts, but we can not find them in literature. For the benefit of the reader, we provide a proof in the Appendix.
\begin{lem}
\label{lem:wellknown}
Let $A$ be a $\mathcal{PS}$--contracting subset of a metric space $X$, then $A$ is strongly quasiconvex.
\end{lem}


 \emph{Flip manifolds} are graph manifolds that are constructed as follows. Take a finite collection of products of $S^1$ with compact orientable hyperbolic surface. Glue them along boundary tori by maps which interchange the basis and fiber direction (see \cite{KapovichLeeb98}). In \cite{Sisto11}, Sisto constructs a path system for the universal cover of a \emph{flip manifold}. We first review Sisto's construction.

 Let $M$ be a flip manifold, we equip $M$ with a nice metric as described in Section~2.2 \cite{KapovichLeeb98}. This metric has the following properties. It is a locally $\CAT(0)$ metric on $M$ and the restriction of this metric on each Seifert component $M_i = F_i \times S^1$ is a hyperbolic metric (such that all boundary components are totally geodesic of unit length) on $F_i$ cross an Euclidean metric on $S^1$. This metric on $M$ induces a metric on $\tilde{M}$, which is denoted by $d$. We note that $(\tilde{M}, d)$ is a $\CAT(0)$ space by Cartan-Hadamard Theorem (see Theorem~4.1 on page~194 in \cite{MR1744486}). Lift the JSJ decomposition of the graph manifold $M$ to the universal cover $\tilde{M}$, and let $T_{\tilde{M}}$ be the tree dual to this decomposition of $\tilde{M}$. Each Seifert piece $M_i$ of $M$ is a product $F_{i} \times S^1$ where $F_i$ is a compact surface with negative Euler characteristic number. Thus each piece $\tilde{M}_i$ of $\tilde{M}$ is the product $\tilde{F}_{i} \times \R$. We will identify $\tilde{F}_{i}$ with $\tilde{F}_{i} \times \{0\}$.

\begin{defn}[Special paths for flip graph manifolds \cite{Sisto18}]
\label{defn:specialpath}
Let $M$ be a flip manifold. A path in $\tilde{M}$ is called a \emph{special path} if it is constructed as follows. For any $x$ and $y$ in $\tilde{M}$. If $x$ and $y$ belong to the same piece $\tilde{M}_i$ of $\tilde{M}$ for some $i$, the special path connecting $x$ to $y$ is defined to be the geodesic from $x$ to $y$. 

We now assume that $x$ and $y$ belong to different pieces of $\tilde{M}$. Let $[x,y]$ be the geodesic in $(\tilde{M},d)$ connecting $x$ to $y$. The path $[x,y]$ passes through a sequence of pieces $\tilde{M}_{0}, \dots, \tilde{M}_{n}$ of $\tilde{M}$ where $x \in \tilde{M}_{0}$, $y \in \tilde{M}_{n}$, $n \ge 1$. 

For convenience, relabel $x$ by $x_0$ and $y$ by $x_{n+1}$. For each $i \in \{0, \dots, n-1\}$. Let $\tilde{T}_{i} = \tilde{M}_{i} \cap \tilde{M}_{i+1}$. The plane $\tilde{T}_{i}$ covers a JSJ torus $T_i$ obtained by identifying boundary tori $\overleftarrow{T_i}$ and $\overrightarrow{T_i}$ of Seifert pieces $M_{i}$ and $M_{i+1}$ of $M$.

For each $i \in \{1, \dots, n-1\}$, let $p_{i}$ and $q_{i}$ be the points in the lines $\tilde{T}_{i-1} \cap \tilde{F}_{i}$ and $\tilde{T}_{i} \cap \tilde{F}_{i}$ respectively such that the geodesic $[p_i, q_i]$ is the shortest path joining two lines $\tilde{T}_{i-1} \cap \tilde{F}_{i}$ and $\tilde{T}_{i} \cap \tilde{F}_{i}$. 

Let $p_{0}$ be the projection of $x_0 \in \tilde{M}_{0} = \tilde{F}_{0} \times \R$ into the base surface $\tilde{F}_{0}$. Let $q_{0}$ be the point in $\tilde{T}_{0} \cap \tilde{F}_{0}$ minimizing the distance from $p_0$. 

Let $q_{n}$ be the projection of $x_{n+1} = y \in \tilde{M}_{n} = \tilde{F}_{n} \times \R$ into the base surface $\tilde{F}_{n}$. Let $p_{n}$ be the point in $\tilde{T}_{n-1} \cap \tilde{F}_{n}$ minimizing the distance from $q_n$.


So far, we have a sequence of points $p_0, q_0, p_1, q_1, \dots, p_n, q_n$. For each $i \in \{0, \dots, n-1\}$, let $\overleftarrow{\ell_i}$ and $\overrightarrow{\ell_i}$ be the Euclidean geodesics in $\tilde{T}_i$ passing through $q_i$ and $p_{i+1}$ such that they project to fibers $\overleftarrow{f_i} \subset \overleftarrow{T_i}$ and $\overrightarrow{f_i} \subset \overrightarrow{T_i}$ respectively. Two lines $\overleftarrow{\ell_i}$ and $\overrightarrow{\ell_i}$ intersects at a point in $\tilde{T}_{i}$, which is denoted by $x_{i+1}$. Hence, we have a sequence of point $x = x_0, x_1, \dots, x_{n+1} =y$. Let $\gamma_i$ be the geodesic connecting $x_i$ to $x_{i+1}$. Let $\gamma$ be the concatenation $\gamma_{0} \cdot \gamma_{1} \cdots \gamma_{n}$. Then $\gamma$ is the \emph{special path} from $x$ to $y$.
\end{defn}


\begin{lem}[Proposition~3.6 \cite{Sisto18}]\label{SistoSPathLem}
Let $M$ be a flip manifold.
Let $\mathcal{PS}(\tilde{M})$ be the collection of the special paths in $\tilde{M}$. Then $\mathcal{PS}(\tilde{M})$ is a path system of $\tilde{M}$.
\end{lem}

Before we get into the proofs of Proposition~\ref{pkey} and Proposition~\ref{actionontree}, we need several lemmas.

\begin{lem}
\label{lem:1}
Let $F$ be a connected compact surface with nonempty boundary and $\chi(F) <0$. Let $M = F \times S^1$. Equip $F$ with a hyperbolic metric and equip $M$ with the product metric. Let $A$ be a subset of $\tilde{M}$ such that $diam(A) \le \delta$ for some $\delta >0$. Then there exists a constant $r>0$ that depends only on $\delta$ and the metric on $F$ such that the following holds.
Let $E$ and $E'$ be two planes boundaries of $\tilde{M}$ such that $A \cap E \neq \emptyset$ and $A \cap E' \neq \emptyset$. Let $\ell$ (resp. $\ell'$) be the boundary line of $\tilde{F}$ such that $\ell \subset E $ (resp. $\ell' \subset E'$). Let $p \in \ell$ and $q \in \ell'$ such that the geodesic $[p, q]$ is the shortest path joining $\ell$ to $\ell'$. For any $x \in E \cap A$ and $y \in E' \cap A$, let $u$ and $v$ be the projection of $x$ and $y$ to the lines $\ell$ and $\ell'$ respectively. Then $d(u, p) \le r$ and $d(v, q) \le r$.
\end{lem}

\begin{proof}
 The chosen hyperbolic metric on $F$ induces a metric on $\tilde{F}$ which is denoted by $d_{\tilde{F}}$. Note that $(\tilde{F}, d_{\tilde{F}})$ is Bilipschitz homeomorphic to a fattened tree (see the paragraph after Lemma~1.1 in \cite{BN08}). Thus, there exists a constant $\epsilon >0$ that depends only on the metric $d_{\tilde{F}})$ such that for any $s \in \ell$ and $t \in \ell'$ we have $$d_{\tilde{F}}(s,p) + d_{\tilde{F}}(p, q) + d_{\tilde{F}}(q, t) \le \epsilon + d_{\tilde{F}}(s,t)$$ 
Since $u \in \ell$, $v \in \ell'$ and $d$ is the product metric of $d_{\tilde{F}}$ with the metric on $\R$, we have
\[
d(u, p) + d(p, q) + d(q, v) \le \epsilon +d(u,v)
\]
Since $diam(A) \le \delta$, it follows that $d(x, y) \le \delta$. Hence $d(u, v) \le \delta$. Let $r = \delta + \epsilon$, it is easy to see that $d(u, p) \le r$ and $d(v, q) \le r$.
\end{proof}

\begin{lem}
\label{lem:2}
Let $M$ be a flip manifold equipped with a metric as described in previous paragraphs. Let $\tilde{K}$ be a subset of $\tilde{M}$ such that the following holds.
\begin{enumerate}
    \item There is a positive constant $\delta$ such that the following holds. If $\tilde{K}$ has nonempty intersection with a piece $\tilde{M}_i$ of $\tilde{M}$, then $\tilde{K} \cap \tilde{M}_i$ is simply connected and $diam(\tilde{K} \cap \tilde{M}_{i}) \le \delta$.
    \item For each plane $E$, the intersection $\tilde{K} \cap E$ is either empty or a disc. The graph $T_{\tilde{K}}$ duals to the decomposition of $\tilde{K}$ along those discs is a subtree of $T_{\tilde{M}}$.
\end{enumerate}
There exists $R >0$ such that the following holds. For any $x, y \in \tilde{K}$, let $\gamma$ be the special path in $\tilde{M}$ connecting $x$ to $y$. Let $\tilde{M}_1, \cdots, \tilde{M}_{s}$ be the sequence of Seifert pieces where $\gamma$ passing through. Then 
\[
\gamma \cap \tilde{M_i} \subset \mathcal{N}_{R}(\tilde{K} \cap \tilde{M}_i).
\]
\end{lem}
\begin{proof}
Given $\delta$, for each Seifert piece $M_i$ of $M$, let $r_i$ be the constant given by Lemma~\ref{lem:1}. Since there are finitely many Seifert pieces of $M$, we can choose a constant $R_1 >0$ so that the conclusion of Lemma~\ref{lem:1} applies to all Seifert pieces of $M$.

Let $R =5R_1 + 5\delta$. We consider the following cases.

Case~1: $x$ and $y$ belong to the same piece $\tilde{M}_i$ of $\tilde{M}$. Since $diam(\tilde{K} \cap \tilde{M}_i) \le \delta$ and $x, y \in \tilde{K} \cap \tilde{M}_i$, we have $d(x, y) \le \delta$. By definition, the special path $\gamma$ connecting $x$ to $y$ is the geodesic connecting $x$ to $y$. Thus $\gamma \subset \mathcal{N}_{\delta}(\tilde{K}) \subset \mathcal{N}_{R}(\tilde{K})$

Case~2: $x$ and $y$ belong to two distinct pieces of $\tilde{M}$. Let $\gamma$ be the special path connecting $x$ to $y$. The path $\gamma$ is constructed explicitly in Definition~\ref{defn:specialpath}. Let $x_0, x_1, \dots, x_n$ be the sequence of points given by Definition~\ref{defn:specialpath}. Let $\gamma_i$ be the geodesic connecting $x_i$ to $x_{i+1}$. We recall that $\gamma$ is the concatenation $\gamma_{0} \cdot \gamma_{1} \cdots \gamma_{n}$.

{\bf Claim:} $\gamma_{i} \subset \mathcal{N}_{R}(\tilde{K})$ for each $i \in \{0, \dots, n\}$.

The proofs of the cases $i =0$ and $i =n$ are similar, so we only need the proof for case $i =0$. The proofs of the cases $i =1, \dots, n-1$ are similar, so we only give the proof for the case $i =1$.

\underline{Proof of the case $i =0$:}

Since $\tilde{K} \cap \tilde{T}_{0} \neq \emptyset$, we choose a point $O_1 \in \tilde{K} \cap \tilde{T}_{0}$. Let $U_1 \in \tilde{T}_{0} \cap \tilde{F}_{0}$ be the projection of $O_1$ into the base surface $\tilde{F}_{0}$ of $\tilde{M}_{0} = \tilde{F}_{0} \times \R$. Let $V_1 \in \tilde{T}_{0} \cap \tilde{F}_{1}$ be the projection of $O_1$ into the base surface $\tilde{F}_{1}$ of $\tilde{M}_{1} = \tilde{F}_{1} \times \R$.

By Lemma~\ref{lem:1}, we have $d(V_1, p_1) \le R_1$. Since $O_1, x_0 \in \tilde{K} \cap \tilde{M}_{0}$ and $diam(\tilde{K} \cap \tilde{M}_{0}) \le \delta$, it follows that $d(O_1, x_0) \le \delta$. Since $p_0$ is the projection of $x_0$ to the base surface $\tilde{F}_{0}$ of $\tilde{M}_0$, and the metric on each piece $\tilde{M}_{i} = \tilde{F}_{i} \times \R$ is the product metric, we have $d(p_0, U_1) \le \delta$. By the construction, $q_0$ is the point in $\tilde{T}_{0} \cap \tilde{F}_{0}$ such that $d(p_0, q_0) \le d(p_0, y)$ for all $y \in \tilde{T}_{0} \cap \tilde{F}_{0}$. Since $U_1 \in \tilde{T}_{0} \cap \tilde{F}_{0}$, it follows that $d(p_0, q_0) \le d(p_0, U_1) \le \delta$. By the triangle inequality, we have $d(U_1, q_0) \le d(U_1, p_0) + d(p_0, q_0) \le \delta + \delta = 2\delta$.
Using Euclidean geometry in the plane $\tilde{T}_0$ we have
\[
d(O_1, x_1) = \sqrt{d(U_1, q_0)^2 + d(V_1, p_1)^2 } \le d(U_1, q_0) + d(V_1, p_1) \le 2\delta + R_1
\]
Since $x_0, O_1$ belong to $\tilde{K} \cap \tilde{M}_0$ and $diam(\tilde{K} \cap \tilde{M}_0) \le \delta$, it implies that $d(O_1, x_0) \le \delta$. Thus 
\[
d(x_0, x_1) \le d(x_0, O_1) + d(O_1, x_1) \le \delta + 2\delta + R_1 = 3\delta + R_1
\]
Since $\gamma_0$ is the geodesic connecting $x_0 \in \tilde{K}$ to $x_1 \in \mathcal{N}_{2\delta + R_1}(O_1)$, it is easy to see that $\gamma_{0} \subset \mathcal{N}_{R}(\tilde{K})$. 

\underline{Proof of the case $i=1$:}
Since $\tilde{K} \cap \tilde{T}_1 \neq \emptyset$, we choose a point $O_2$ in $\tilde{K} \cap \tilde{T}_1 \neq \emptyset$. Let $U_2 \in \tilde{T}_{1} \cap \tilde{F}_{1}$ be the projection of $O_2$ to the base surface $\tilde{F}_{1}$ of $\tilde{M}_{1}$. Let $V_2 \in \tilde{T}_{1} \cap \tilde{F}_{2}$ be the projection of $O_2$ to the base surface $\tilde{F}_{2}$ of $\tilde{M}_{2}$.

 By Lemma~\ref{lem:1}, we have $d(q_1, U_2) \le R_1$ and $d(p_2, V_2) \le R_1$. Using Euclidean geometry in the plane $\tilde{T}_1$ we have 
\[
d(x_2, O_2) = \sqrt{d(q_1, U_2)^2 + d(p_2, V_2)^2} \le d(q_1, U_2) +d(p_2, V_2) \le 2R_1
\]
Since $O_1, O_2 \in \tilde{K} \cap \tilde{M}_{1}$, it follows that $d(O_1, O_2) \le \delta$. Thus $d(x_1, x_2) \le d(x_1, O_1) + d(O_1, O_2) + d(O_2, x_2) \le 2\delta + R_1 + \delta + 2R_1 = 3\delta + 3R_1$.

Since $\gamma_1$ is a geodesic connecting $x_1$ to $x_2$, it is easy to see that $\gamma_1 \subset \mathcal{N}_{R}(\tilde{K})$.

Since $\gamma \cap \tilde{M}_{i} = \gamma_{i}$, the lemma is proved.
\end{proof}

\begin{lem}
\label{lem:contracting}
Let $M$ be a flip manifold and let $\mathcal{PS}(\tilde{M})$ be the collection of special paths in $\tilde{M}$.
Let $\tilde{K}$ be the set given by Lemma~\ref{lem:2}. Then
$\tilde{K}$ is $\mathcal{PS}$--contracting subset of $\tilde{M}$.
\end{lem}

\begin{proof}
Let $\delta$ and $R$ be the constants given by Lemma~\ref{lem:2}. Let $C =10\delta + R$.
First, we are going to define a $\mathcal{PS}$--projection $\pi \colon \tilde{M} \to \tilde{K}$ on $\tilde{K}$. 

Let $\pi' \colon T_{\tilde{M}} \to T_{\tilde{K}}$ be the projection from the tree $T_{\tilde{M}}$ to the subtree $T_{\tilde{K}}$.

For any $x \in \tilde{M}$, there exists a piece $\tilde{M}_i$ of $\tilde{M}$ such that $x \in \tilde{M}_i$. The piece $\tilde{M}_i$ is corresponding to a vertex, denoted by $v(x)$, in $T_{\tilde{M}}$. Then $\pi'(v(x))$ is a vertex of $T_{\tilde{K}} \subset T_{\tilde{M}}$. Note that $\tilde{K} \cap \tilde{M}_{\pi'(v(x))} \neq \emptyset$ and $diam(\tilde{K} \cap \tilde{M}_{\pi'(v(x))}) \le \delta$. Choose $\pi(x)$ to be a point in $\tilde{K} \cap \tilde{M}_{\pi'(v(x))}$.

By the construction of $\pi$, if $x$ is a point in $\tilde{K}$, then $d(x, \pi(x)) \le \delta < C$. Hence, the map $\pi$ satisfies the condition (1) of Definition~\ref{defn:contracting}. 

We are going to verify (2) of Definition~\ref{defn:contracting}. Let $x$ and $y$ be two points in $\tilde{M}$ such that $d(\pi(x), \pi(y)) \ge C$. Let $\gamma$ be the special path in $\tilde{M}$ connecting $x$ to $y$. We want to show that the distance from $\pi(x)$ and $\pi(y)$ to $\gamma$ is no more than $C$. Let $\kappa$ be the number of Seifert pieces of $\tilde{M}$ where the geodesic $[\pi(x), \pi(y)]$ in $\tilde{M}$ is traveling through. Then we have $d(\pi(x), \pi(y)) \le \kappa \delta$. Indeed, we call these Seifert pieces by $\tilde{M}_{1}, \dots, \tilde{M}_{\kappa}$. Note that $\tilde{K} \cap \tilde{M}_{i} \neq \emptyset$ with $i \in \{1, \dots, \kappa\}$. Choose a point $s_i$ in $\tilde{M}_{i} \cap \tilde{M}_{i+1} \cap \tilde{K}$. We have that $d(\pi(x), s_1) \le \delta$, $d(\pi(y), s_{\kappa-1}) \le \delta$ and $d(s_{i}, s_{i+1}) \le \delta$ with $i \in \{1, \dots, \kappa-2 \}$. Using triangle inequality, we have $d(\pi(x), \pi(y)) \le \kappa \delta$. We now have
\[
10\delta < C \le d(\pi(x), \pi(y)) \le \kappa \delta
\]
Hence, $10 < \kappa$. This shows that the distance of two vertices $\pi'(v(x))$ and $\pi'(v(y))$ is at least $10$.

Choose the geodesic in the tree $T_{\tilde{K}}$ connecting $\pi'(v(x))$ to $\pi'(v(y))$. Choose vertices $v_i$ and $v_j$ in this geodesic such that the distance between two vertices $v_i$ and $\pi'(v(x))$ is $3$ and the distance between two vertices $v_j$ and $\pi'(v(y))$ is $3$. Let $\beta$ be the geodesic in the tree $T_{\tilde{K}}$ connecting $v_i$ to $v_j$. Let $\alpha$ be the special path in $\tilde{M}$ connecting $\pi(x) \in \tilde{M}_{\pi'(v(x))}$ to $\pi(y) \in \tilde{M}_{\pi'(v(y))}$. We remark here there is a subpath of $\gamma$ (and thus this subpath is also a special path) such that this path and $\alpha$ connecting a point in $\tilde{M}_{\pi'(v(x))}$ to a point in $\tilde{M}_{\pi'(v(y))}$. By Remark~3.3 in \cite{Sisto11}, we have two paths $\alpha$ and $\gamma$ coincide in pieces $\tilde{M}_{v}$ where $v$ is any vertex of $\beta$. 

Applying Lemma~\ref{lem:2} to the path $\alpha$, we have $\alpha \subset \mathcal{N}_{R}(\tilde{K})$. In particular, for any vertex $v$ of $\beta$ we have $\alpha \cap \tilde{M}_{v} \subset \mathcal{N}_{R}(\tilde{K} \cap \tilde{M}_{v})$. Since $v_i$ is a vertex of $\beta$, we choose a point $u \in \alpha \cap \tilde{M}_{v_i}$ and a point $u' \in \tilde{K} \cap \tilde{M}_{v_i}$ so that $d(u, u') \le R$. Since $\pi(x) \in \tilde{K} \cap \tilde{M}_{\pi'(v(x))}$, $u' \in \tilde{K} \cap \tilde{M}_{v_i}$ and the distance between $v_i$ and $\pi'(v(x))$ in the tree is $3$, we have $d(\pi(x), u') \le 4\delta$. Thus
\[
d(\pi(x), u) \le d(\pi(x), u') + d(u', u) \le 4\delta + R <C
\]
Since $\alpha$ and $\gamma$ coincide in $\tilde{M}_{v_i}$, it follows that $u \in \gamma$. Thus, $d(\pi(x), \gamma) \le d(\pi(x), u) < C$. Similarly, we can show that $d(\pi(y), \gamma) < C$. Therefore, the theorem is proven.
\end{proof}


\begin{prop}
\label{pkey}
Let $M$ be a graph manifold group. Let $H$ be a finitely generated purely Morse subgroup of $\pi_1(M)$. Then $H$ is strongly quasiconvex.
\end{prop}

\begin{proof}
We equip $M$ with a Riemannian metric. By Theorem~2.3 in \cite{KapovichLeeb98}, there exists a nonpositively curved flip-manifold $N$ and a bilipschitz homeomorphism $\phi \colon \tilde{M} \to \tilde{N}$ such that $\phi$ preserves their geometric decompositions.


Let $M_H \to M$ be the covering space corresponding to the subgroup $H \le \pi_1(M)$. Let $K$ be the compact Scott core of $M_H$ given by Subsection~\ref{subsubsection:1}. Let $\tilde{K}$ be the preimage of $K$ in the universal cover $\tilde{M}$. Using Lemma~\ref{lem:scottcoreuniversal} together with the fact $\phi$ is bilipschitz homeomorphism, we have that the image $\phi(\tilde{K}) \subset \tilde{N}$ satisfies (1) and (2) of Lemma~\ref{lem:2}. By Lemma~\ref{lem:contracting}, $\phi(\tilde{K})$ is $\mathcal{PS}(\tilde{N})$--contracting. Thus, $\phi(\tilde{K})$ is strongly quasiconvex in $\tilde{N}$ by Lemma~\ref{lem:wellknown}. It follows that $\tilde{K}$ is strongly quasiconvex in $\tilde{M}$. As a consequence, $H$ is strongly quasiconvex in $\pi_1(M)$.


\end{proof}

\begin{prop}
\label{actionontree}
Let $M$ be a graph manifold. Let $H$ be a finitely generated subgroup of $\pi_1(M)$. Then the following are equivalent:
\begin{enumerate}
    \item $H$ is purely Morse in $\pi_1(M)$;
    \item The action of $H$ on the Bass–Serre tree of $M$ induces a quasi-isometric embedding from $H$ into the tree.
\end{enumerate}
\end{prop}

\begin{proof}
The implication ``$(2)\implies (1)$'' is straight forward. In fact, if some nontrivial element $g$ in $H$ is not Morse, then $g$ is conjugate into a Seifert piece subgroup by Proposition~\ref{psis}. Therefore, $g$ fixes a vertex of the Bass–Serre tree of $M$ which contradicts to Statement (2).

Now we prove the implication ``$(1)\implies (2)$''.

Let $M_H \to M$ be the covering space corresponding to the subgroup $H \le \pi_1(M,x_0)$. Without loss of generality, we can assume that the base point $x_0$ belongs to the interior of some Seifert piece of $M$. Let $\tilde{x}_0$ be a lift point of $x_0$. Let $K$ be the Scott core of $M_H$ given by Section~\ref{subsubsection:1}. Let $\tilde{M}$ be the universal cover of $M$, and let $\tilde{K}$ be the preimage of $K$ in the universal cover $\tilde{M}$. 

Since all nontrivial elements in $H$ are Morse in $\pi_1(M)$, we recall that all pieces $M_{H,i}$ of $M^{c}_H$ are simply connected (see Lemma~\ref{lem:M_H,isimplyconnected}). The point $\tilde{x}_0$ belong to a Seifert piece $\tilde{M}_0$ of $\tilde{M}$, and this Seifert piece corresponds to a vertex in the tree $T_{\tilde{M}}$, denoted by $v$. We first define a map $\Phi \colon H \to T_{\tilde{M}}$ as the following. For any $h \in H$, the point $h \cdot \tilde{x}_0$ belong to a Seifert piece of $\tilde{M}$, and this Seifert piece corresponds to a vertex in $T_{\tilde{M}}$ that we denote by $\Phi(h)$.

Since $H$ acts on $T_{\tilde{M}}$ by isometry and the map $\Phi \colon H \to T_{\tilde{M}}$ is $H$-equivariant, we only need to show that there exist $L \ge 1$ and $C \ge 0$ such that for any $h \in H$ we have \[
\frac{1}{L}d_{H}(e,h) - C \le d_{T_{\tilde{M}}}(v,\Phi(h)) \le Ld_{H}(e,h) + C
\] 

By Proposition~\ref{pkey}, $H$ is strongly quasiconvex in $\pi_1(M)$. Hence $H$ is undistorted in $\pi_1(M)$.
It follows that there exists a positive number $\epsilon > 1$ such that for any $h, k \in H$ such that
\begin{equation}
\tag{$\clubsuit$}
\label{eqn1}
   \frac{1}{\epsilon}d(h\cdot \tilde{x}_0, k \cdot \tilde{x}_0) - \epsilon \le d_{H}(h,k) \le \epsilon d(h \cdot \tilde{x}_0, k \cdot \tilde{x}_0) + \epsilon
\end{equation}

Let $\delta$ be the constant given by Lemma~\ref{lem:scottcoreuniversal}. Let $\rho$ be the minimum distance of any two distinct JSJ planes in $\tilde{M}$. We have that the following property holds. Let $x \neq y$ be two points in $\tilde{K} \subset \tilde{M}$. Let $[x,y]$ be a geodesic in $\tilde{M}$ connecting $x$ to $y$, and let $n$ be the number of Seifert pieces of $\tilde{M}$ which $[x,y]$ passes through. Then 
\begin{equation}
\tag{$\diamondsuit$}
\label{eqn2}
   (n-2)\rho \le d(x,y) \le \delta n
\end{equation}

Let $L =\frac{\epsilon}{\rho} + \delta \epsilon$ and $C = \frac{\epsilon}{\delta} + 2 + \frac{\epsilon^2}{\rho}$.


For each $h \in H$, if $h =e$, then there is nothing to show. We consider the case that $h$ is nontrivial element of $H$. It follows that $\Phi(h) \neq v$ (because each piece $M_{H,i}$ is simply connected). Let $n$ be the number of Seifert pieces of $\tilde{M}$ which $[\tilde{x}_0, h \cdot \tilde{x}_0]$ passes through. We have
\[
d_{T_{\tilde{M}}}(\Phi(h), v) = n -1
\]

Using (\ref{eqn2}) we have
\[
(n-2) \rho \le d(\tilde{x}_0, h \cdot \tilde{x}_0) \le n\delta
\]
Hence,
\[
\frac{1}{\delta}d(\tilde{x}_0, h \cdot \tilde{x}_0) -1 \le d_{T_{\tilde{M}}}(\Phi(h), v) \le 1 + \frac{1}{\rho}d(\tilde{x}_0, h \cdot \tilde{x}_0)
\]
Combining with (\ref{eqn1}) we have
\[
\frac{1}{\delta\epsilon} d_{H}(e,h) - \frac{1}{\delta} -1 \le d_{T_{\tilde{M}}}(\Phi(h), v) \le \frac{\epsilon}{\rho}d_{H}(e,h) + 1 + \frac{\epsilon^2}{\rho}
\]
Thus, \[
\frac{1}{L}d_{H}(e,h) - C \le d_{T_{\tilde{M}}}(v,\Phi(h)) \le Ld_{H}(e,h) + C
\] for all $h \in H$. In other words, $\Phi \colon H \to T_{\tilde{M}}$ is a quasi-isometric embedding. 
\end{proof}

\section{Strongly quasiconvex subgroups of $3$--manifold groups}
\label{sec:sqcinfgmanifold}

In this section, we complete the proof of Theorem~\ref{thm:introduction2}. 
We remark here that we already proved this theorem when $M$ is a graph manifold in Section~\ref{sscfgmg}. 
We now prove the theorem for the case of geometric manifold $M$ except Sol in Section~\ref{subsection:fggeometricmanifold}. In this section, we also show that Theorem~\ref{thm:introduction2} is not true for the case of Sol manifolds. In Section~\ref{sqcfmmf}, we prove Theorem~\ref{thm:introduction2} for the case of mixed manifold $M$ which completes the theorem for the case of nongeometric $3$-manifold $M$. Finally, we complete the theorem for the case of compact, orientable $3$-manifold $M$ 
that does not have a Sol $3$--manifold as a summand in its sphere-disc decomposition in Section~\ref{subsection:fhin3manifold}. 

\subsection{Strongly quasiconvex subgroups of geometric $3$--manifolds}
\label{subsection:fggeometricmanifold}
We recall that a compact orientable irreducible $3$--manifold $M$ with empty or tori boundary is called \emph{geometric} if its interior admits a geometric structure in the sense of Thurston which are $3$--sphere, Euclidean $3$--space, hyperbolic $3$--space, $S^{2} \times \R$, $\mathbb{H}^{2} \times \R$, $\widetilde{SL}(2, \R)$, Nil and Sol.

In the following lemma, we show that all strongly quasiconvex subgroups of Sol $3$--manifold groups are either trivial or of finite index in their ambient groups.

\begin{lem}[Strongly quasiconvex subgroups in Sol $3$--manifold groups]
\label{lem:sqSol}
Let $M$ be a Sol $3$--manifold and let $H$ be a strongly quasiconvex subgroup of $\pi_1(M)$. Then $H$ is trivial or has finite index in $\pi_1(M)$. 
\end{lem}

\begin{proof}
Let $N$ be the double cover of $M$ that is a torus bundle with Anosov monodromy $\Phi$. Then $\pi_1(N)$ is an abelian-by-cyclic subgroup $\field{Z}^2\rtimes_{\Phi} \field{Z}$ and it has finite index subgroup in $\pi_1(M)$. Therefore, each strongly quasiconvex subgroups of $\pi_1(N)$ is trivial or has finite index in $\pi_1(N)$ by Corollary~\ref{cor:nsq}. This implies that $H$ is trivial or has finite index in $\pi_1(M)$.
\end{proof}

In the following lemma, we characterize all finite height subgroups of Sol $3$--manifold groups.

\begin{lem}[Finite height subgroups in Sol $3$--manifolds]
\label{lem:fhSol}
Let $M$ be a Sol $3$--manifold. Let $N$ be the double cover of $M$ that is a torus bundle with Anosov monodromy. Let $H$ be a nontrivial finitely generated infinite index subgroup of $\pi_1(M)$. Then $H$ has finite height in $\pi_1(M)$ if and only if $H \cap \pi_1(N)$ is an infinite cyclic subgroup generated by an element in $\pi_1(N)$ that does not belong to the fiber subgroup of $\pi_1(N)$. 
\end{lem}
\begin{proof}
We note that $\pi_1(N)$ is the semi-direct product $\Z^2 \rtimes_\phi \Z$ where $\phi \in GL_{2}(\Z)$ is the matrix corresponding to the Anosov monodromy. We note that $\phi$ is conjugate to a matrix of the form $ \Bigl(\begin{matrix}
a&b \\ c&d
\end{matrix} \Bigr)$ where $ad -bc =1$ and $|a +d| >2$. By Example~\ref{exmp:1}, we have that $\phi^{\ell}$ has no nontrivial fixed point for any nonzero integer $\ell$.

We are going to prove necessity. Assume that $H$ has finite height in $\pi_1(M)$. It is straightforward to see that $H \cap \pi_1(N)$ has infinite index in $\pi_1(N)$ and $H \cap \pi_1(N)$ is not trivial.
Since $H$ has finite height in $\pi_1(M)$ and $\pi_1(N)$ is a subgroup of $\pi_1(M)$, it follows that $H \cap \pi_1(N)$ has finite height in $\pi_1(N)$ (see Proposition~\ref{p0}). By Proposition~\ref{thm:introduction1}, the subgroup $H \cap \pi_1(N)$ is an infinite cyclic subgroup generated by an element that does not belong to the fiber subgroup of $\pi_1(N)$.

We are going to prove sufficiency. Suppose that $H \cap \pi_1(N)$ is an infinite cyclic subgroup generated by an element in $\pi_1(N)$ that does not belong to the fiber subgroup of $\pi_1(N)$. By Proposition~\ref{thm:introduction1}, the subgroup $H \cap \pi_1(N)$ has finite height in $\pi_1(N)$. Since $\pi_1(N)$ has finite index in $\pi_1(M)$, it follows from (2) of Proposition~\ref{p0} that $H$ has finite height in $\pi_1(M)$.

\end{proof}

In the following lemma we study strongly quasiconvex subgroups and finite height subgroups of the fundamental group $\pi_1(M)$, where $M$ is a $3$--manifold $M$ which has a geometric structure modeled on six of eight geometries: $S^3$, $\R^{3}$, $S^{2} \times \R$, Nil, $\widetilde{SL(2,\mathbb{R})}$ or $\field{H}^2\times \field{R}$.

\begin{lem}
\label{lem:fhgeom}
Suppose that a $3$--manifold $M$ has a geometric structure modeled on six of eight geometries: $S^3$, $\E^{3}$, $S^{2} \times \R$, Nil, $\widetilde{SL(2,\mathbb{R})}$ or $\field{H}^2\times \field{R}$. Let $H$ be a nontrivial, finitely generated subgroup of $\pi_1(M)$. Then the following statements are equivalent:
\begin{enumerate}
    \item $H$ has finite height in $\pi_1(M)$.
    \item $H$ is a finite index subgroup of $\pi_1(M)$.
    \item $H$ is strongly quasiconvex in $\pi_1(M)$.
\end{enumerate}
\end{lem}

\begin{proof}
By Theorem~\ref{tran}, we have (3) implies (1). It is obvious that (2) implies (3). For the rest of the proof, we only need to show that (1) implies (2). 

If the geometry of $M$ is spherical, then its fundamental group is finite; hence it is obvious that $H$ has finite index in $\pi_1(M)$. 

If the geometry of $M$ is either $S^2 \times \field{R}$, $\field{E}^3$, $\textit{Nil}$, $\widetilde{SL(2,\mathbb{R})}$ or $\field{H}^2\times \field{R}$, then $\pi_1(M)$ has a finite index subgroup $K$ such that the centralizer $Z(K)$ of $K$ is infinite. More precisely, 
\begin{enumerate}
    \item If the geometry of $M$ is $S^2 \times \R$, then there exists a finite index subgroup $K \le \pi_1(M)$ such that $K$ is isomorphic to $\Z$. It is obvious that the centralizer $Z(K)$ is infinite.
    \item If the geometry of $M$ is $\mathbb{E}^3$, then there exists a finite index subgroup $K \le \pi_1(M)$ such that $K$ is isomorphic to $\mathbb{Z}^3$. Hence the centralizer $Z(K)$ is infinite.
    \item If the geometry of $M$ is $\textit{Nil}$, then $\pi_1(M)$ contains a discrete Heisenberg subgroup of finite index $K$ which has infinite centralizer.
    \item If the geometry of $M$ is $\mathbb{H}^2 \times \R$, then $M$ is finitely covered by $M' = \Sigma \times S^1$ where $\Sigma$ is a compact surface with negative Euler characteristic. Therefore, $K=\pi_1(M')$ has finite index in $\pi_1(M)$ and $Z(K)$ is infinite.
    \item If $M$ has a geometry modeled on $\widetilde{SL(2,\R)}$, then $M$ is finitely covered by a circle bundle over surface $M'$. Thus $K=\pi_1(M')$ has finite index in $\pi_1(M)$. We note that $Z(K)$ is infinite since it contains the fundamental group of a regular fiber of $M'$.
\end{enumerate}
By Proposition~\ref{p0}, the subgroup $H\cap K$ has finite height in the subgroup $K$. Therefore by Proposition~\ref{central}, the subgroup $H\cap K$ is trivial or has finite index in the subgroup $K$. It follows that $H$ has finite index in $\pi_1(M)$.
\end{proof}

The rest of this section is devoted to the study of strongly quasiconvex subgroups and finite height subgroups of hyperbolic $3$--manifold groups.

\begin{rem}
\label{rem:fhclosedhyp}
If $M$ is a closed hyperbolic $3$--manifold, it is well-known that a finitely generated subgroup $H$ has finite height in $\pi_1(M)$ if and only if $H$ is strongly quasiconvex (equivalently, $H$ is geometrically finite). Indeed, by Subgroup Tameness Theorem, any finitely generated subgroup of $\pi_1(M)$ is either geometrically finite or virtual fiber surface subgroup. If $H$ has finite height in $\pi_1(M)$, then $H$ must be geometrically finite, otherwise $H$ is a virtual fiber surface subgroup that is not a finite height subgroup. Since $M$ is closed hyperbolic $3$--manifold, it is well-known that $\pi_1(M)$ is a hyperbolic group. Since $H$ is geometrically finite, it follows that $H$ is strongly quasiconvex in $\pi_1(M)$. Conversely, if $H$ is strongly quasiconvex in $\pi_1(M)$, then $H$ has finite height by a result of \cite{GMRS98}.
\end{rem}

\begin{prop}
\label{prop:fhtorihyp}
Let $M$ be a hyperbolic $3$--manifold with tori boundary. Let $H$ be a finitely generated subgroup of $\pi_1(M)$. Let $M_H \to M$ be the covering space corresponding to the subgroup $H \le \pi_1(M)$. Then the following statements are equivalent.
\begin{enumerate}
    \item $H$ has finite height in $\pi_1(M)$.
    \item $\partial M_H$ consists only planes, tori
    \item $H$ is strongly quasiconvex in $\pi_1(M)$.
\end{enumerate}
\end{prop}
\begin{proof}
We first prove that (1) implies (2). Let $T_1, T_2, \dots, T_n$ be the tori boundary of $M$. We note that $\pi_1(M)$ is hyperbolic relative to the collection $\PP = \{\pi_1(T_1), \dots, \pi_1(T_n) \}$. Since $M_H \to M$ is a covering spaces, it follows that $\partial M_H$ contains only tori, planes and cylinders.
If $H$ has finite height in $\pi_1(M)$, then by Proposition~\ref{p0} $H \cap g \pi_1(T_i) g^{-1}$ has finite height in $g \pi_1(T_i) g^{-1}$ for any $g \in \pi_1(M)$ and $i \in \{1, \dots, n\}$. Since $g \pi_1(T_i) g^{-1}$ is isomorphic to $\Z^2$, it follows from Proposition~\ref{central} that $H \cap g \pi_1(T_i) g^{-1}$ is trivial or has finite index in $g \pi_1(T_i) g^{-1}$. Thus, $\partial M_H$ does not contain a cylinder.

Now, we are going to prove that (2) implies (3). Assume that $\partial M_H$ consists only tori and planes. It follows that $H$ is geometrically finite subgroup of $\pi_1(M)$. Indeed, if not, then $H$ is geometrically infinite. Hence $M_H$ is homeomorphic to $\Sigma_{H} \times \R$ or a twisted $\R$--bundle $\Sigma_{H} \tilde{\times} \R$ for some compact surface $\Sigma_H$ (with nonempty boundary). Thus, $\partial M_{H}$ contains cylinders that contradicts to our assumption.
Since $H$ is geometrically finite, it follows that $H$ is undistorted in $\pi_1(M)$. Moreover, since $\partial M_H$ consists only planes and tori, it follows that the intersection of $H$ with each conjugate $gPg^{-1}$ of peripheral subgroup in $\PP$ is either trivial or has finite index in $gPg^{-1}$. Thus $H \cap gPg^{-1}$ is strongly quasiconvex in $gPg^{-1}$. By Theorem~\ref{relhyp111}, it follows that $H$ is strongly quasiconvex in $\pi_1(M)$. By Theorem~\ref{tran}, we have (3) implies (1). Therefore, the theorem is proved. 
\end{proof}

\begin{rem}[Morse elements in geometric $3$--manifold groups]
\label{rem:morseingeomanifold}
Let $M$ be a geometric $3$--manifold. If $M$ has a geometric structure modeled on six of eight geometries: $S^3$, $\R^{3}$, Nil, Sol, $\widetilde{SL(2,\mathbb{R})}$ or $\field{H}^2\times \field{R}$, then $\pi_1(M)$ has no Morse element by Lemma~\ref{lem:sqSol} and Lemma~\ref{lem:fhgeom}. If $M$ has a geometric structure modeled $S^2 \times \field{R}$, then $\pi_1(M)$ is virtually an infinite cyclic subgroup. Therefore, all infinite order elements of $\pi_1(M)$ are Morse. Now we consider the case $M$ has a geometric structure modeled on $\field{H}^3$. If $M$ is a closed manifold, then all infinite order elements of $\pi_1(M)$ are Morse. If $M$ is a hyperbolic $3$--manifold with tori boundary, then we let $T_1, T_2, \dots, T_n$ be the tori boundary of $M$. We note that $\pi_1(M)$ is hyperbolic relative to the collection $\PP = \{\pi_1(T_1), \dots, \pi_1(T_n) \}$. Therefore, an infinite order element $g$ in $\pi_1(M)$ is Morse if and if $g$ does not conjugate to an element of $\pi_1(T_i)$ (see Proposition~\ref{Morseelementinrelhyp}).
\end{rem}

\subsection{Strongly quasiconvex subgroups of mixed $3$--manifold groups}
\label{sqcfmmf}

In this section, we prove Theorem~\ref{thm:introduction2} and Theorem~\ref{coolformanifold} for the case of mixed $3$--manifold $M$ which completes the proof of these theorems for the case of nongeometric $3$--manifold groups. 

\begin{prop}
\label{prop:themixedmanifoldcase}
Let $M$ be a mixed $3$--manifold. Let $H$ be a finitely generated subgroup of $\pi_1(M)$. Then $H$ has finite height if and only if $H$ is strongly quasiconvex.
\end{prop}
\begin{proof}
The sufficiency is proved by Theorem~\ref{tran}. We are going to prove the necessity. We will assume that $H$ has infinite index in $\pi_1(M)$, otherwise it is trivial. Let $M_1, \cdots, M_k$ be the maximal graph manifold components of the JSJ decomposition of $M$, let $S_1, \dots, S_{\ell}$ be the tori in the boundary of $M$ that adjoint a hyperbolic piece, and let $T_1, \dots,T_m$ be the tori in the JSJ decomposition of $M$ that separate two hyperbolic components of the JSJ decomposition. Then $\pi_1(M)$ is hyperbolic relative to $\PP = \{\pi_1(M_p)\} \cup \{\pi_1(S_q)\} \cup \{\pi_1(T_r)\}$ (see \cite{Dah03} or \cite{BW13}).

By Theorem~\ref{relhyp111}, to see that $H$ is strongly quasiconvex in $\pi_1(M)$, we only need to show that $H$ is undistorted in $\pi_1(M)$ and $H \cap gPg^{-1}$ is strongly quasiconvex in $gPg^{-1}$ for each conjugate $gPg^{-1}$ of peripheral subgroup in $\PP$. In the rest of the proof, we are going to verify the following claims.

\underline{Claim~1:} $H$ is undistorted in $\pi_1(M)$. 

Indeed, let $p \colon M_H \to M$ be the covering space corresponding to the subgroup $H \le \pi_1(M)$. Let $M^{c}_{H} \subset M_{H}$ be the submanifold of $M_H$ given by Subsection~\ref{subsubsection:1}. For each piece $M_{H,i}$ of $M^{c}_{H}$, let $M_i$ be a piece of $M$ such that $M_{H,i}$ covers $M_i$. Since $H$ has finite height in $\pi_1(M)$, it follows from (1) in Proposition~\ref{p0} that $\pi_1(M_{H,i})$ has finite height in $\pi_1(M_i)$. If $M_i$ is a Seifert piece of $M$, then $\pi_1(M_{H,i})$ is either finite or has finite index in $\pi_1(M_i)$. If $M_i$ is a hyperbolic piece of $M$, then $\pi_1(M_{H,i})$ is strongly quasiconvex in $\pi_1(M_i)$ (see Proposition~\ref{prop:fhtorihyp}), hence $\pi_1(M_{H,i})$ is geometrically finite in $\pi_1(M_i)$. Thus, the ``almost fiber surface'' $\Phi(H)$ of the subgroup $H \le \pi_1(M)$ (see Section~3.1 in \cite{Sun18}) is empty. It follows from Theorem~1.4 in \cite{NS19} that $H$ is undistorted in $\pi_1(M)$.

\underline{Claim~2:} $H \cap gPg^{-1}$ is strongly quasiconvex in $gPg^{-1}$ for each conjugate $gPg^{-1}$ of peripheral subgroup in $\PP$.

We first show that each subgroup $H \cap gPg^{-1}$ is finitely generated. In fact, since $gPg^{-1}$ is strongly quasiconvex and $H$ is undistorted, then $H \cap gPg^{-1}$ is strongly quasiconvex in $H$ by Proposition 4.11 in \cite{Tran2017}. This implies that $H \cap gPg^{-1}$ is finitely generated. We now prove that $H \cap gPg^{-1}$ is strongly quasiconvex. Since $H$ has finite height in $\pi_1(M)$, it follows that $H \cap gPg^{-1}$ has finite height in $gPg^{-1}$ (see Proposition~\ref{p0}). If $P$ is either $\pi_1(S_q)$ or $\pi_1(T_r)$ for some $\pi_1(S_q), \pi_1(T_r) \in \PP $, then $gPg^{-1}$ is isomorphic to $\Z^2$. By Proposition~\ref{central}, $H \cap gPg^{-1}$ is either finite or has finite index in $gPg^{-1}$. It implies that $H \cap gPg^{-1}$ is strongly quasiconvex in $gPg^{-1}$. If $P$ is $\pi_1(M_j)$ for some maximal graph manifold component $M_j$, then $H \cap gPg^{-1}$ is strongly quasiconvex in $gPg^{-1}$ by Theorem~\ref{coolforgraphmanifold}.



\end{proof}

The following proposition is the direct result of Proposition~\ref{Morseelementinrelhyp} and Proposition~\ref{psis}.

\begin{prop}[Morse elements in mixed manifold groups]
\label{psis1}
Let $M$ be a mixed $3$--manifold group. Then a nontrivial group element $g$ in $\pi_1(M)$ is Morse if and only if $g$ is not conjugate into any Seifert subgroups. 
\end{prop}

The following proposition is the direct result of Theorem~\ref{coolforgraphmanifold} and Proposition~\ref{corcoolcool}.

\begin{prop}[Purely Morse subgroups in mixed manifold groups]
\label{ptran}
Let $M$ be a mixed $3$--manifold group and let $H$ be an undistorted purely Morse subgroup of $\pi_1(M)$. Then $H$ is stable in $\pi_1(M)$. 
\end{prop}


\subsection{Strongly quasiconvex subgroups of finitely generated $3$--manifold groups}
\label{subsection:fhin3manifold}
In Section~\ref{sscfgmg}, Section~\ref{subsection:fggeometricmanifold} and Section~\ref{sqcfmmf}, we have shown that finitely generated finite height subgroups and strongly quasiconvex subgroups are equivalent in the fundamental group of a $3$--manifold with empty or tori boundary (except Sol $3$--manifold). In this subsection, we extend these results to arbitrary finitely generated $3$--manifold groups.

\begin{prop}
\label{prop:finiteheightinhigherboundary}
Let $M$ be a compact orientable irreducible $3$--manifold that has trivial torus decomposition and has at least one higher genus boundary component. Let $H$ be a finitely generated subgroup of $\pi_1(M)$.

Then $H$ has finite height in $\pi_1(M)$ if and only if $H$ is strongly quasiconvex.
\end{prop}
\begin{proof}
If $H$ has finite index in $\pi_1(M)$, then the result is obviously true. Hence, we will assume that $H$ has infinite index in $\pi_1(M)$. The sufficiency is followed from Theorem~\ref{tran}. We are going to prove necessity.

We paste hyperbolic $3$--manifolds with totally geodesic boundaries to $M$ to get a finite volume hyperbolic $3$--manifold $N$ (for example, see Section~6.3 in \cite{Sun18}).

Claim: $\pi_1(M)$ is strongly quasiconvex in $\pi_1(N)$. Indeed, by the construction of $N$, the subgroup $\pi_1(M) \le \pi_1(N)$ is not a virtual fiber, hence it is geometrically finite. Thus, $\pi_1(M)$ is undistorted. If $N$ is closed then $\pi_1(N)$ is a hyperbolic group. Since $\pi_1(M)$ is undistorted in $\pi_1(N)$, it follows that $\pi_1(M)$ is strongly quasiconvex in $\pi_1(N)$. As a consequence, $\pi_1(M)$ has finite height in $\pi_1(N)$ (see \cite{GMRS98}). If $N$ has nonempty boundary, then $N$ has tori boundary, we let $N_M \to N$ be the covering space corresponding to the subgroup $\pi_1(M) \le \pi_1(N)$. We note that $\partial N_{M}$ does not contain any cylinder. By Proposition~\ref{prop:fhtorihyp}, we have $\pi_1(M)$ is strongly quasiconvex in $\pi_1(N)$. The claim is established.

We are now going to show that $H$ is strongly quasiconvex in $\pi_1(M)$ if $H$ has finite height in $\pi_1(M)$. Indeed, by the above claim, $\pi_1(M)$ has finite height in $\pi_1(N)$. Since $H$ has finite height in $\pi_1(M)$, it follows that $H$ has finite height in $\pi_1(N)$ (see Proposition~\ref{p0}). Applying Remark~\ref{rem:fhclosedhyp} and Proposition~\ref{prop:fhtorihyp} to the hyperbolic manifold $N$ (in the case $\partial N = \emptyset$ and $\partial N \neq \emptyset$ respectively), we have $H$ is strongly quasiconvex in $\pi_1(N)$. Since $\pi_1(M)$ is undistorted in $\pi_1(N)$, it follows from Proposition~4.11 in \cite{Tran2017} that $H = H \cap \pi_1(M)$ is strongly quasiconvex in $\pi_1(M)$. 
\end{proof}

We now prove Proposition~\ref{pintro} and Theorem~\ref{thm:introduction2}.

\begin{proof}[Proof of Proposition~\ref{pintro}]
Since $M$ is a compact, orientable $3$--manifold, it decomposes into irreducible, $\partial$--irreducible pieces $M_1, \dots, M_k$ (by the sphere-disc decomposition). In particular, $\pi_1(M)$ is the free product $\pi_1(M_1)* \pi_1(M_2)* \cdots *\pi_1(M_k)$. Let $G_i = \pi_1(M_i)$. We remark here that $\pi_1(M)$ is hyperbolic relative to the collection $\PP = \{G_1, \cdots, G_k \}$.

 Suppose that $M_j$ is a Sol manifold for some $j \in \{1, 2, \dots, k\}$. By Lemma~\ref{lem:sqSol} and Lemma~\ref{lem:fhSol}, there exists a finitely generated, finite height subgroup $A$ of $\pi_1(M_j)$ such that $A$ is not strongly quasiconvex in $\pi_1(M_j)$.
 
Since each peripheral subgroup in a relatively hyperbolic group is strongly quasiconvex, it follows that $\pi_1(M_j)$ has finite height in $\pi_1(M)$ by Theorem~\ref{tran}. Since $A$ has finite height in $\pi_1(M_j)$ and $\pi_1(M_j)$ has finite height in $\pi_1(M)$, it follows from Proposition~\ref{p0} that $A$ has finite height in $\pi_1(M)$.

We note that $A$ is not strongly quasiconvex in $\pi_1(M)$. Indeed, by way of contradiction, suppose that $A$ is strongly quasiconvex in $\pi_1(M)$. Since $\pi_1(M_j)$ is undistorted in $\pi_1(M)$, it follows from Proposition~4.11 in \cite{Tran2017} that
that $A = A \cap \pi_1(M_j)$ is strongly quasiconvex in $\pi_1(M_j)$. This contradicts to the fact that $A$ is not strongly quasiconvex in $\pi_1(M_j)$. 
\end{proof}

\begin{proof}[Proof of Theorem~\ref{thm:introduction2}]

By Theorem~\ref{tran}, if $H$ is strongly quasiconvex in $\pi_1(M)$, then $H$ has finite height in $\pi_1(M)$. Thus, for the rest of the proof, we only need to show that $H$ is strongly quasiconvex in $\pi_1(M)$ if $H$ has finite height. We also assume that $H$ has infinite index in $\pi_1(M)$, otherwise the result is vacuously true.



We can reduce to the case that $M$ is irreducible and $\partial$--irreducible by the following reason:
Since $M$ is compact, orientable $3$--manifold, it decomposes into irreducible, $\partial$--irreducible pieces $M_1, \dots, M_k$ (by the sphere-disc decomposition). In particular, $\pi_1(M)$ is a free product $\pi_1(M_1)* \pi_1(M_2)* \cdots *\pi_1(M_k)$. Let $G_i = \pi_1(M_i)$. We remark here that $\pi_1(M)$ is hyperbolic relative to the collection $\PP = \{G_1, \cdots, G_k \}$. 
By Kurosh Theorem, the subgroup $H \cong H_{1} * \cdots H_{m} *F_k$ where each subgroup $H_i = H \cap g_{i}G_{i_j}g^{-1}_{i}$ for some $g_{i} \in \pi_1(M)$, and $i_j \in \{1, \dots, k\}$. We remark here that $H$ is strongly quasiconvex in $\pi_1(M)$ if finitely generated finite height subgroups of $G_i$ are strongly quasiconvex. Indeed, to see this, we note that $H_i$ has finite height in $g_{i}G_{i_j}g^{-1}_{i}$ since $H$ has finite height in $\pi_1(M)$. It follows that $H_i$ is strongly quasiconvex in $g_{i}G_{i_j}g^{-1}_{i}$ (by the assumption above), and hence $H_i$ is undistorted in $g_{i}G_{i_j}g^{-1}_{i}$. The argument in the second paragraph of the proof of Theorem~1.3 in \cite{NS19} tells us that $H$ must be undistorted in $\pi_1(M)$. Using Theorem~\ref{relhyp111}, we have that $H$ is strongly quasiconvex in $\pi_1(M)$. We note that by the hypothesis, none of pieces $M_1, \dots, M_k$ supports the Sol geometry.
Therefore, for the rest of the proof we only need to show that finitely generated finite height subgroups in the fundamental group of a compact, orientable, irreducible, $\partial$--irreducible manifold that does not support the Sol geometry are strongly quasiconvex.


We consider the following cases.

{\bf Case~1}: $M$ has trivial torus decomposition.

Case~1.1: $M$ has empty or tori boundary. In this case, $M$ has a geometric structure modeled on seven of eight geometries: $S^3$, $\R^{3}$, $S^{2} \times \R$, Nil, $\widetilde{SL(2,\mathbb{R})}$, $\field{H}^2\times \field{R}$, $\mathbb{H}^{3}$. By Lemma~\ref{lem:fhgeom}, Remark~\ref{rem:fhclosedhyp}, and Proposition~\ref{prop:fhtorihyp}, we have that $H$ is strongly quasiconvex in $\pi_1(M)$.

Case~1.2: $M$ has higher genus boundary. In this case, it follows from Proposition~\ref{prop:finiteheightinhigherboundary} that $H$ is strongly quasiconvex.

{\bf Case~2}: $M$ has nontrivial torus decomposition.

Case~2.1: $M$ has empty or tori boundary. Then $M$ is a nongeometric $3$--manifolds. By Theorem~\ref{coolforgraphmanifold} and Proposition~\ref{prop:themixedmanifoldcase}, $H$ is strongly quasiconvex.

Case~2.2: $M$ has a boundary component of genus at least $2$. By Section~6.3 in \cite{Sun18}, we paste hyperbolic $3$--manifolds with totally geodesic boundaries to $M$ to get a $3$--manifold $N$ with empty or tori boundary. The new manifold $N$ satisfies the following properties.
\begin{enumerate}
    \item $M$ is a submanifold of $N$ with incompressible tori boundary.
    \item The torus decomposition of $M$ also gives the torus decomposition of $N$.
    \item Each piece of $M$ with a boundary component of genus at least $2$ is contained in a hyperbolic piece of $N$.
\end{enumerate}
We remark here that it has been proved in \cite{NS19} that $\pi_1(M)$ is undistorted in $\pi_1(N)$ (see the proof of Case~1.2 in the proof of Theorem~1.3 in \cite{NS19}). 

Claim: $\pi_1(M)$ is strongly quasiconvex in $\pi_1(N)$. 

We are going to prove the claim above. Let $M'_{1}, \dots, M'_{k}$ be the pieces of $M$ that satisfies (3).
Since $N$ is a mixed $3$--manifold, we equip $N$ with a nonpositively curved metric as in \cite{Leeb95thesis}. This metric induces a metric on the universal cover $\tilde{N}$. Let $N'_i$ be the hyperbolic piece of $N$ such that $M'_i$ is contained in $N'_i$.
Note that by Proposition~\ref{prop:fhtorihyp}, the subgroup $\pi_1(M'_i)$ is strongly quasiconvex in $\pi_1(N'_i)$. 
Since there are only finitely many pieces $M'_{1}, \dots, M'_{k}$, it follows that for any $K \ge 1$, $C \ge 0$, there exists $R = R (K,C)$ such that for any $(K,C)$--quasi-geodesic in $\tilde{N}'_v$ with endpoints in $\tilde{M}'_{v}$ (for some $\tilde{M}'_{v} \subset \tilde{M}$ covers a $M'_i \subset N'_i$) then this quasi-geodesic lies in the $R$--neighborhood of $\tilde{M}'_v$.

Let $\gamma$ be any $(K,C)$--quasi-geodesic in $\tilde{N}$. By Taming quasi-geodesic Lemma (see Lemma~1.11 page 403 in \cite{MR1744486}), we can assume that $\gamma$ is a continuous path. By the construction of the manifold $N$, the path $\gamma$ can be decomposed into a concatenation $\gamma = \alpha_1 \cdot \beta_1 \cdot \alpha_2 \cdot \beta_2 \cdots \alpha_{\ell} \cdot \beta_{\ell} \cdot \alpha_{\ell+1}$ such that for each $j$, the subpath $\alpha_j$ is a subset of $\tilde{M}$, and $\beta_j$ intersects $\tilde{M}$ only at its endpoints. Here $\alpha_1$ and $\alpha_{\ell+1}$ might degenerate to points. Moreover, there are pieces $\tilde{M}'_j$ and $\tilde{N}'_j$ of $\tilde{M}$ and $\tilde{N}$ respectively such that $\tilde{M}'_j \subset \tilde{N}'_j$, $\beta_{j} \subset \tilde{N}'_j$, and the endpoints of $\beta_j$ in $\tilde{M}'_j$.

 We have that $\beta_j \subset \mathcal{N}_{R}(\tilde{M}'_j) \subset \mathcal{N}_{R}(\tilde{M})$. Thus $\gamma \subset \mathcal{N}_{R}(\tilde{M})$. Therefore, $\tilde{M}$ is strongly quasiconvex in $\tilde{N}$. It follows that $\pi_1(M)$ is strongly quasiconvex in $\pi_1(N)$. 
Since $H$ has finite height in $\pi_1(M)$ and $\pi_1(M)$ has finite height in $\pi_1(N)$, it follows that $H$ has finite height in $\pi_1(N)$ by Proposition~\ref{p0}. Since $N$ is a nongeometric $3$--manifold, it follows from Theorem~\ref{coolforgraphmanifold} and Proposition~\ref{prop:themixedmanifoldcase} that $H$ is strongly quasiconvex in $\pi_1(N)$. Since $\pi_1(M)$ is undistorted in $\pi_1(N)$, it follows from Proposition~4.11 \cite{Tran2017} that $H = H \cap \pi_1(M)$ is strongly quasiconvex in $\pi_1(M)$.
\end{proof}

\begin{rem}
\label{rem:fg}
In this remark, we explain how to reduce the study on finite height subgroups and strongly quasiconvex subgroups from all finitely generated $3$--manifold groups to the case of compact, orientable $3$--manifold groups.

Since $\pi_1(M)$ is a finitely generated group, it follows from the Scott core theorem that $M$ contains a compact codim--$0$ submanifold such that the inclusion map of the submanifold into $M$ is a homotopy equivalence. In particular, the inclusion induces an isomorphism on their fundamental groups. We note that $H$ has finite height (resp. is strongly quasiconvex) in $\pi_1(M)$ if and only if the preimage of $H$ under the isomorphism has finite height (resp. is strongly quasiconvex) in the fundamental group of the submanifold. Thus, we can assume that the manifold $M$ is compact.

We can also assume that $M$ is orientable. Indeed, let $M'$ be a double cover of $M$ that is orientable. We remark here that a finitely generated subgroup $H$ of $\pi_1(M)$ has finite height if and only if $H': = H \cap \pi_1(M')$ has finite height in $\pi_1(M')$. Moreover, $H$ is strongly quasiconvex in $\pi_1(M)$ if and only if $H'$ is strongly quasiconvex in $\pi_1(M')$. Therefore, without loss of generality, we can assume that $M$ is orientable.
\end{rem}

We are now ready for the proof of Theorem~\ref{coolformanifold}.
\begin{proof}[Proof of Theorem~\ref{coolformanifold}]
Using the similar argument as in Remark~\ref{rem:fg} and in the proof of Theorem~\ref{thm:introduction2}, we can reduce the theorem to the case where $M$ is compact, orientable, irreducible and $\partial$--irreducible. We leave it to the interesting reader. We consider the following cases:

Case~1: $M$ has trivial torus decomposition. 

Case~1.1: $M$ has empty or torus boundary. In this case, $M$ has geometric structure modeled on $S^3$, $\R^{3}$, $S^{2} \times \R$, Nil, $\widetilde{SL(2,\mathbb{R})}$, $\field{H}^2\times \field{R}$, $\mathbb{H}^{3}$. By Remark~\ref{rem:morseingeomanifold}, if the geometry of $M$ is $S^3$, $\R^{3}$, Nil, $\widetilde{SL(2,\mathbb{R})}$, $\field{H}^2\times \field{R}$, then $H$ must be finite. Thus $H$ is stable. If the geometry of $M$ is $S^{2} \times \R$, then from the fact $\pi_1(M)$ is virtually $\Z$, it follows that $H$ is stable. If the geometry of $M$ is $\mathbb{H}^3$, then $H$ is stable by the following reasons. If $M$ is closed, then $\pi_1(M)$ is hyperbolic group. Since $H$ is undistorted in $\pi_1(M)$, it follows that $H$ is a hyperbolic group and $H$ is strongly quasiconvex in $\pi_1(M)$. Hence $H$ is stable. If $M$ has tori boundary, then $H$ is stable by combination of Remark~\ref{rem:morseingeomanifold} and Theorem~\ref{relhyp112}.

Case~1.2: $M$ has higher genus boundary. Let $N$ be the finite volume hyperbolic $3$--manifold constructed in the proof of Proposition~\ref{prop:finiteheightinhigherboundary}.
It follows from Case~1.1 above that all undistorted purely Morse subgroups of $\pi_1(N)$ are stable. By Statement (3) of Corollary~\ref{ccool}, we have that all undistorted purely Morse subgroups of $\pi_1(M)$ are stable. Thus, $H$ is stable in $\pi_1(M)$.

Case~2: $M$ has nontrivial torus decomposition. If the geometry of $M$ is Sol, then $H$ is trivial by Remark~\ref{rem:morseingeomanifold}. Thus, we can assume that $M$ does not support the Sol geometry.

Case~2.1: $M$ has empty or tori boundary. In this case, $M$ is either a graph manifold or a mixed manifold. If $M$ is a graph manifold, then $H$ is stable by Theorem~\ref{coolforgraphmanifold} (see (2) implies (4)). If $M$ is a mixed manifold, then it follows from Proposition~\ref{ptran} that $H$ is stable.

Case~2.2: $M$ has higher genus boundary. Let $N$ be the mixed $3$--manifold constructed in Case~2.2 of the proof of Theorem~\ref{thm:introduction2}. We recall that we have shown in the proof of Theorem~\ref{thm:introduction2} that $\pi_1(M)$ is strongly quasiconvex in $\pi_1(N)$. Since $\pi_1(N)$ has the property that all undistorted purely Morse subgroups of $\pi_1(N)$ are stable (see Case~2.1 above), it follows that all undistorted purely Morse subgroups of $\pi_1(M)$ are stable by Statement (3) of Corollary~\ref{ccool}. Thus $H$ is stable in $\pi_1(M)$.
\end{proof}


\section*{Appendix A}
\subsection*{Finite height subgroups, malnormal subgroups, and strongly quasiconvex subgroups of $\field{Z}^k\rtimes_{\Phi} \field{Z}$}
\label{sec:semidirectproduct}
In this section, we study strongly quasiconvex subgroups and finite height subgroups of abelian-by-cyclic subgroups $\field{Z}^k\rtimes_{\Phi} \field{Z}$. First, we define a \emph{fixed point} of a group automorphism $\Phi:\!\field{Z}^k \to \field{Z}^k$ is a group element $z$ in $\field{Z}^k$ such that $\Phi(z)=z$. The main result of this section is the following proposition.
 


\begin{propappend}
\label{thm:introduction1}
Let $\Phi:\!\field{Z}^k \to \field{Z}^k$ be a group automorphism. Then the group $G=\field{Z}^k\rtimes_{\Phi} \field{Z}=\langle\field{Z}^k,t|tzt^{-1}=\Phi(z)\rangle$ has a finite height subgroup which is not trivial and has infinite index if and only if for every non-zero integer $\ell$, the group automorphism $\Phi^{\ell}$ has no nontrivial fixed point. Moreover, a nontrivial, infinite index subgroup $H$ has finite height if and only if $H$ is a cyclic subgroup generated by a group element $g\in G-\field{Z}^k$.
\end{propappend}

The proof of Proposition~\ref{thm:introduction1} is a combination of Lemma~\ref{p1}, Lemma~\ref{p2}, and Lemma~\ref{p3} as follows.

\begin{lemappend}
\label{p1}
Let $G=\field{Z}^k\rtimes_{\Phi} \field{Z}=\langle\field{Z}^k,t|tzt^{-1}=\Phi(z)\rangle$ and $H$ a nontrivial subgroup of infinite index of $G$. Assume that $H$ is a finite height subgroup. Then $H$ is a cyclic subgroup generated by $t^{m}z$ where $m$ is a positive integer and $z$ is an element in $\field{Z}^k$.
\end{lemappend}
\begin{proof}
It follows from Corollary~\ref{cor0} that that $H\cap \field{Z}^k$ is trivial. Thus $H$ is a cyclic subgroup generated by $t^{m}z$ where $m$ is a positive integer and $z$ is an element in $\field{Z}^k$.
\end{proof}

\begin{corappend}[Strongly quasiconvex subgroups $\Longrightarrow$ are trivial or have finite index]
\label{cor:nsq}
Let $G=\field{Z}^k\rtimes_{\Phi} \field{Z}=\langle\field{Z}^k,t|tzt^{-1}=\Phi(z)\rangle$ and $H$ a strongly quasiconvex subgroup of $G$. Then either $H$ is trivial or $H$ has finite index in $G$.
\end{corappend}

\begin{proof}
We observe that the group $G$ is a solvable group. By \cite{MR2153979}, none of the asymptotic cones of $G$ has a global cut-point. Also by \cite{MR2584607}, the group $G$ does not contain any Morse element.

Assume that $H$ is not trivial and has infinite index in $G$. Then $H$ is a finite height subgroup by Theorem~\ref{tran}. By Proposition~\ref{p1}, $H$ is a cyclic subgroup generated by $t^{m}z$ where $m$ is a positive integer and $z$ is an element in $\field{Z}^k$. Therefore, $G$ contains the Morse element $t^{m}z$ which is a contradiction. Thus, either $H$ is trivial or $H$ has finite index in $G$.
\end{proof}

\begin{lemappend}
\label{p2}
Let $\Phi:\!\field{Z}^k \to \field{Z}^k$ be a group automorphism such that $\Phi^{\ell}$ has a nontrivial fixed point $z_0$ for some non-zero integer $\ell$. Let $H$ be a finite height subgroup of $G=\field{Z}^k\rtimes_{\Phi} \field{Z}=\langle\field{Z}^k,t|tzt^{-1}=\Phi(z)\rangle$. Then either $H$ is trivial or $H$ has finite index in $G$.
\end{lemappend}

\begin{proof}
Assume that $H$ is not trivial and has infinite index in $G$. By Proposition~\ref{p1}, the subgroup $H$ is a cyclic subgroup generated by $t^{m}z$ where $m$ is a positive integer and $z$ is an element in $\field{Z}^k$. Since $\Phi^{\ell}(z_0)=z_0$, the group element $z_0$ commutes to the group element $t^\ell$. Therefore, $z_0$ commutes to the group element $(t^{m}z)^\ell$ in $H$. Therefore, $\bigcap\limits_{i=1}^{\infty} z^{i}_0Hz^{-i}_0$ is infinite. Also $z^{i}_0H\neq z^{j}_0H$ for each $i\neq j$. Therefore, $H$ is not a finite height subgroup of $G$ which is a contradiction. Therefore, either $H$ is trivial or $H$ has finite index in $G$.
\end{proof}

\begin{lemappend}
\label{p3}
Let $\Phi:\!\field{Z}^k \to \field{Z}^k$ be a group automorphism such that $\Phi^{\ell}$ has no nontrivial fixed point for every non-zero integer $\ell$. Let $G=\field{Z}^k\rtimes_{\Phi} \field{Z}=\langle\field{Z}^k,t|tzt^{-1}=\Phi(z)\rangle$ and $H$ be the cyclic subgroup of $G$ generated by $t^{m}z$ where $m$ is a positive integer and $z$ is an element in $\field{Z}^k$. Then $H$ has height at most $m$. In particular, any cyclic group generated by $tz$ where $z\in \field{Z}^k$ is malnormal.
\end{lemappend}

\begin{proof}
Assume that $H$ does not have height at most $m$. Then there are $(m+1)$ distinct left cosets $g_1H, g_2H, g_3H,\cdots, g_{m+1}H$ such that $\bigcap\limits_{i=1}^{m+1}g_i H g^{-1}_i$ is infinite. We observer that there are $i\neq j$ such that $g=g^{-1}_ig_j$ can be written of the form $t^{mq}z_1$ for some integer $q$ and some group element $z_1\in\field{Z}$. Since $g=t^{mq}z_1$ is not a group element in $H$, we can write $g=z_0 (t^mz)^q$ for some group element $z_0\in \field{Z}-\{e\}$.

Since the subgroup $g_iHg^{-1}_i\cap g_jHg^{-1}_j$ is infinite, the subgroup $H\cap gHg^{-1}$ is also infinite. Therefore, there is a non-zero integer $p$ such that $g(t^mz)^pg^{-1}=(t^mz)^p$. Also, $g=z_0 (t^mz)^q$ for some group element $z_0\in \field{Z}-\{e\}$. Thus, $z_0(t^mz)^pz_0^{-1}=(t^mz)^p$. It is straight forward that $(t^mz)^p=t^{mp}z'$ for some $z'\in \field{Z}^k$. Therefore, $z_0(t^{mp}z')z_0^{-1}=t^{mp}z'$. This implies that $z_0(t^{mp})z_0^{-1}=t^{mp}$. Thus, $\Phi^{mp}(z_0)=t^{mp}z_0t^{-mp}=z_0$ which is a contradiction. Therefore, $H$ is a finite height subgroup.
\end{proof}

\begin{exmpappend}
\label{exmp:1}
Let $\Phi \colon \Z^2 \to \Z^2$ be an automorphism such that its corresponding matrix has the form $ \Bigl(\begin{matrix}
a&b \\ c&d
\end{matrix} \Bigr)$ where $ad -bc =1$ and $|a+d| >2$. We note that $\phi$ has two real eigenvalues $\lambda$ and $1/\lambda$ such that $\lambda \neq 1,-1$ and $Trace(\Phi) = a+ d = \lambda + 1/ \lambda$. For any non-zero integer $\ell$, the automorphism $\Phi^{\ell}$ has two eigenvalues $\lambda^{\ell}$ and $1/(\lambda)^{\ell}$ which have absolute value $\neq 1$. Hence $\Phi^{\ell}$ has no nontrivial fixed point. Another way to see is that $\Phi^{\ell}$ has the form of $ \Bigl(\begin{matrix}
a'&b' \\ c'&d'
\end{matrix} \Bigr)$ where $a'd' -b'c' =1$. Since $Trace(\Phi^{\ell}) = a'+d' = \lambda^{\ell} + 1/(\lambda)^{\ell}$, it follows that $|a' +d'| >2$. It easy to see that the matrix $ \Bigl(\begin{matrix}
a'&b' \\ c'&d'
\end{matrix} \Bigr)$ has no nontrivial fixed point (otherwise $|a'+d'|=2$). By Proposition~\ref{thm:introduction1} and Corollary~\ref{cor:nsq}, the group $Z^{2} \rtimes_\Phi \Z$ has a finite height subgroup $H$ which is not strongly quasiconvex.
\end{exmpappend}

\subsection*{The $\mathcal {PS}$ system and strong quasiconvexity}
In this section, we first give the proof for the statement that all special paths for a flip graph manifold are uniform quasi-geodesic. This fact seems not to be proved explicitly in \cite{Sisto11} and \cite{Sisto18}. Then we give the proof for Lemma~\ref{lem:wellknown} which states that a $\mathcal {PS}$--contracting subset is strongly quasiconvex. 


\begin{propappend}
\label{uniformquasi-geodesic}
The special path for a flip graph manifold is uniform quasi-geodesic.
\end{propappend}


Let $M$ be a flip manifold. Let $\tilde{M}_0 = \tilde{F}_{0} \times \R$ and $\tilde{M}_1 = \tilde{F}_{1} \times \R$ be two adjacent pieces in $\tilde{M}$ with a common boundary $\tilde{T}_0 = \tilde{M}_{0} \cap \tilde{M}_{1}$. Since $M$ is a flip manifold, it follows that the boundary line $\overrightarrow{\ell_1}: = (\tilde{F}_{0} \times \{0\}) \cap \tilde{T}_{0}$ of $\tilde{F}_{0}$ projects to a fiber in $M_1$ and the boundary line $\overleftarrow{\ell_1}: = (\tilde{F}_{1} \times \{0\}) \cap \tilde{T}_{0}$ of $\tilde{F}_1$ projects to a fiber in $M_0$.

By abuse of language, we denote by $d_h$ the hyperbolic distance on $\tilde F_i$, and $d_v$ the fiber distance of $\tilde M_i$ for $i=0,1$. However, the following fact is crucial: the $d_h$-distance in $\tilde M_0$ on the boundary $\overrightarrow{\ell_1}$ of $\tilde F_0$ coincides with the $d_v$-distance on the fiber $\overleftarrow{\ell_1}$ of $\tilde M_1$.

Let $\delta=[x,y][y,z]$ be a concatenated path of geodesics $[x, y]$ and $[y,z]$ where $x=(x^h,x^v)\in \tilde M_0$, $y=(y^h,y^v)=(y^v, y^h)\in \tilde M_0\cap \tilde M_1=\tilde T_0$, and $z=(z^h,z^v)\in \tilde M_1$. Note that the coordinates of $y$ in $\tilde M_0$ and $\tilde M_1$ are switched.

Consider a \textit{minimizing horizontal slide} of $y$ in $\tilde M_0$ which changes its $\tilde F_0$-coordinate only so that the projection of $[x,y]$ on $\tilde F_0$ is orthogonal to $\overrightarrow{\ell_1}$. To be precise, a minimizing horizontal slide in $\tilde M_0$ applied to $y$ gives a point $w=(w^h,y^v)\in \tilde{T_0} \cap \tilde M_0$ with the same $\overleftarrow{\ell_1}$--coordinate as $y$ so that $d_h(x^h, w^h)$ minimizes the distance $d_h(x^h, \overrightarrow{\ell_1})$.

We need the following observation that a minimizing horizontal slide does not increase distance too much. In what follows, we will work with $L^1$-metric on $\tilde M$ and denote by $|x-y|$ by the $L^1$-distance from $x$ to $y$.
\begin{lemappend}[Horizontal slide]\label{HSildeLem}
There exists a constant $C>0$ depending only on $M$ with the following property. If $w=(w^h,y^v)\in \tilde{T_1} \cap \tilde M_0$ is a point so that $d(x^h, w^h)$ minimizes the distance of $d(x^h, \overrightarrow{\ell_1})$, then $$|x-w|+|w-z|\le |x-y| +|y-z|+ C$$
\end{lemappend}

\begin{proof}
We have
$$|x-w| + |w-z| =  d_h(x^h, w^h)+d_h(y^v, z^h)+d_v(x^v, y^v)+d_v(w^h, z^v)
$$
and
$$
|x-y| + |y-z| = d_h(x^h, y^h)+d_h(y^v, z^h)+d_v(x^v, y^v)+d_v(y^h, z^v).
$$
Hence, it suffices to find a constant $C$ depending only on $M$ such that the following inequality holds.
\begin{equation}\label{L1DistEQ}
d_h(x^h, w^h) +d_v(w^h, z^v) 
\le d_h(x^h, y^h)+d_v(y^h, z^v)+C.
\end{equation}
By assumption, $w^h$ is a shortest projection point of $x^h$ to $\overrightarrow{\ell_1}$. By hyperbolicity of $\tilde F_0$, since $y^h$ and $w^h$ lie on $ \overrightarrow{\ell_1}$, we have 
\begin{equation}\label{HypDistEQ}
d_h(x^h, w^h)+d_h(w^h, y^h)\le d_h(x^h, y^h) + C
\end{equation}
for some constant $C>0$ depending only on hyperbolicity constant of $\tilde F_0$. 
Since the $d_h$-distance in $\tilde M_0$ on the boundary $\overrightarrow{\ell_1}$ of $\tilde F_0$ coincides with the $d_v$-distance of $\tilde M_1$ on the fiber $\overrightarrow{\ell_1}$, we then obtain $$d_h(w^h, y^h)=d_v(w^h, y^h)=|d_v(w^h, z^v)-d_v(y^h, z^v)|,$$
which with (\ref{HypDistEQ}) together proves (\ref{L1DistEQ}). 
The lemma is thus proved.
\end{proof}

\begin{proof}[Proof of the Proposition~\ref{uniformquasi-geodesic}]
We use the notion $a \asymp b$ if there exists $K = K(M)$ such that $a/ K \le b \le Ka$.

We follow the notations in Definition~\ref{defn:specialpath}. Let $\gamma=\gamma_0\cdot\gamma_1 \cdots \gamma_n$ be a special path between $x = x_0 \in \tilde M_0$ and $y= x_{n+1}\in \tilde M_{n}$ so that $\gamma_i=[x_i, x_{i+1}]\subset \tilde M_i$ where $x_{i+1}\in \tilde T_i:=\tilde M_i\cap \tilde M_{i+1}$ for $0\le i \le n-1$. Then we have 
$$
\ell(\gamma) = \sum_{i=0}^n d(x_i, x_{i+1}) \asymp \sum_{i=0}^n |x_i- x_{i+1}|
$$
where $\asymp$ follows from the fact that the $L^1$-metric is bi-lipschitz equivalent to a $L^2$-metric on each piece.

Let $\delta$ be the $\CAT(0)$ geodesic with same endpoints as $\gamma$. Let $y_i$ be the intersection point of $\delta$ with $\tilde T_{i-1}$. We let $y_0: = x_0$ and $y_{n+1} := x_{n+1}$. Then
$$
\ell(\delta) = \sum_{i=0}^n d(y_i, y_{i+1})\asymp \sum_{i=0}^n |y_i- y_{i+1}|
$$
We apply a sequence of minimizing horizontal slides to the endpoints of geodesics $[y_i, y_{i+1}]$ in order to transform the geodesic $\delta$ to the special path $\gamma$. More precisely, we first apply a horizontal slide (in $\tilde{M}_0$) of $y_1\in\tilde T_0$ to $w_1$ so by Lemma \ref{HSildeLem}, we have 
\[
|y_0 - w_1| + |w_1 - y_2| \le |y_0 - y_1| + |y_1 -y_2| + C
\]
and then apply a horizontal slide of $w_1$ to $x_1$ in $\tilde{M}_1$ we have
\[
|y_2 - x_1| + |x_1 - y_0| \le |y_2 -w_1| + |w_1 -y_0| + C
\]
Using the fact $y_0 = x_0$ and two inequalities above, we have
\[
|x_0 -x_1| = |y_0 - x_1| \le |y_0 -x_1|+ |x_1 -y_2| \le |y_0-y_1|+ |y_1 -y_2| +2C 
\]
Inductively, we apply horizontal slides of $y_i$ to $w_i$ in the piece $\tilde M_{i-1}$ and then $w_i$ to $x_i$ in the piece $\tilde M_{i}$ to get
\[
\sum_{i=0}^{n}|x_i -x_{i+1}| \le 2\sum_{i=0}^{n}|y_i -y_{i+1}| + 2C(n+1)
\]
Let $\rho>0$ be the minimal distance of any two JSJ planes. Since $x_0 \in \tilde{M}_0$ and $x_{n+1} \in \tilde{M}_{n+1}$, it follows that $(n-1)\rho \le d(x,y)$. Hence
\begin{align*}
\sum_{i=0}^{n}|x_i -x_{i+1}| &\le 2\sum_{i=0}^{n}|y_i -y_{i+1}| + 2C(n+1)\\
&=2\sum_{i=0}^{n}|y_i -y_{i+1}| + 2C(n-1) + 4C\\
&\le 2\sum_{i=0}^{n}|y_i -y_{i+1}| + 2C d(x,y)/\rho + 4C
\end{align*}
Therefore, we showed that there exists $\kappa >0$ such that for any special path $\gamma$ in $\tilde{M}$ we have $\ell(\gamma) \le \kappa d(\gamma_{+}, \gamma_{-}) + \kappa$, where $\gamma_{+}$ and $\gamma_{-}$ are endpoints of $\gamma$. The proposition is proved. 
\end{proof}

Now we give a proof for Lemma~\ref{lem:wellknown}. Before giving the proof, we need the following fact. Recall that all paths in the $\mathcal {PS}$ system are $c$--quasi-geodesic for some uniform constant $c$.

\begin{lemappend}\cite[Lemma 2.4(3,4)]{Sisto18}\label{BallProjLem}
Let $\pi$ be a $\mathcal {PS}$-projection with constant $C$ on $A \subset X$. Then there exists a constant $k=k(c, C)>0$ with the following properties: 
\begin{enumerate}
    \item For each $x \in X$ we have $diam(\pi(B_r(x))) \le C$ for $r =
d(x, A)/k - k$.
    \item For each $x\in X$ we have $d(x,\pi(x))\leq kd(x,A)+k.$
\end{enumerate}

\end{lemappend}

\begin{proof}[Proof of the Lemma~\ref{lem:wellknown}]
Let $\gamma$ be a $(\lambda, \lambda)$-quasi-geodesic with two endpoints in $A$, where $\lambda\ge 1$. Let $\bar{c}=\max\{k, \lambda\}$.
For a constant $R>0$, we consider a connected component $\alpha$ of $\gamma\setminus N_R(A)$ with initial and terminal endpoints $\alpha_-, \alpha_+$. We need at most $\bar c \ell(\alpha)/(R-\bar c^2)$ balls of radius $R/\bar c-\bar c$ to cover $\alpha$, where $\ell(\alpha)$ denotes the length of $\alpha$. 

We now set $R=\bar c^2(1+2C).$
On the one hand, we obtain by Lemma \ref{BallProjLem}(1) that $$diam(\pi(\alpha))\le \bar c C \ell(\alpha)/(R-\bar c^2)\le \ell(\alpha)/2\bar c.$$

On the other hand, since $\alpha$ is a $(\lambda,\lambda)$-quasi-geodesic for $\lambda\le \bar c$, we have
\begin{align*}
\ell(\alpha)&\le \lambda d(\alpha_-,\alpha_+)+\lambda   \\
& \le \bar c(d(\alpha_-, \pi(\alpha_-))+d(\alpha_+, \pi(\alpha_+))+diam(\pi(\alpha)))+\bar c\\
& \le \bar c(2\bar cR+2\bar c+ \ell(\alpha)/2\bar c)+\bar c, 
\end{align*} 
where the last two inequalities follow from a projection estimate combined with Lemma \ref{BallProjLem}(2). We thus obtain $\ell(\alpha)\le 4\bar c^2 (R+2)$. As a consequence, we have that $\gamma$ is contained in the $4\bar c^2(R+2)$-neighborhood of $A$. This proves the lemma.
\end{proof}

\bibliographystyle{alpha}
\bibliography{FH}

\end{document}